\documentclass[a4paper]{article}
\usepackage[utf8]{inputenc}
\usepackage{amssymb}
\usepackage{amsmath}
\usepackage{hyperref}
\usepackage{amsmath}
\usepackage{amsfonts}
\usepackage{amsthm}
\usepackage{graphicx}

\newtheorem{theorem}{Theorem}
\newtheorem{corollary}[theorem]{Corollary}
\newtheorem{observation}[theorem]{Observation}

\newtheorem{conjecture}[theorem]{Conjecture}

\newtheorem{lemma}[theorem]{Lemma}

\newtheorem{megaclaim}{Claim}
\newcommand{\claim}[2]{\begin{megaclaim}\label{#1} #2 \end{megaclaim}}
\newcommand{\refclaim}[1]{Claim~\ref{#1}}
\newcommand{\cin}{\mbox{int}}
\newcommand{\cex}{\mbox{ext}}
\newcommand{\defi}{\mbox{def}}
\newcommand{\GG}{\mathcal{G}}

\newenvironment{subproof}{%
  \begin{proof}[Subproof]%
}{%
  \end{proof}%
}

\title{$(3a:a)$-list-colorability of embedded graphs of girth at least five}
\author{Zden\v{e}k Dvo\v{r}\'{a}k\thanks{Computer Science Institute (CSI) of Charles University,
           Malostransk\'{e} n\'am\v{e}st\'{\i} 25, 118 00 Prague,
	              Czech Republic. E-mail: \protect\href{mailto:rakdver@iuuk.mff.cuni.cz}{\protect\nolinkurl{rakdver@iuuk.mff.cuni.cz}}.
	Supported by project 17-04611S (Ramsey-like aspects of graph coloring) of Czech Science Foundation.}\and 
	Xiaolan Hu\thanks{School of Mathematics and Statistics $\&$ Hubei Key Laboratory of Mathematical Sciences, Central China Normal University,
	Wuhan 430079, PR China.  Partially supported by NSFC under grant number 11601176 and NSF of Hubei Province under grant
	number 2016CFB146.}
}
\date{}

\begin{document}
\maketitle

\begin{abstract}
A graph $G$ is \emph{list $(b:a)$-colorable} if for every assignment of lists
of size $b$ to vertices of $G$, there exists a choice of an
$a$-element subset of the list at each vertex such that
the subsets chosen at adjacent vertices are disjoint.  We prove that for every positive integer $a$,
the family of minimal obstructions of girth at least five to list $(3a:a)$-colorability is strongly hyperbolic, in the sense
of the hyperbolicity theory developed by Postle and Thomas.  This has a number of consequences, e.g.,
that if a graph of girth at least five and Euler genus $g$ is not list $(3a:a)$-colorable, then $G$ contains a
subgraph with $O(g)$ vertices which is not list $(3a:a)$-colorable.
\end{abstract}

While it is NP-hard to decide whether a planar graph is $3$-colorable~\cite{garey1979computers},
Grötzsch~\cite{grotzsch1959} proved that every planar triangle-free graph is $3$-colorable.  These facts motivated
the development of a rich theory of $3$-colorability of triangle-free embedded graphs.
Thomassen~\cite{thomassen-surf} proved that there are only finitely many $4$-critical graphs
(minimal obstructions to $3$-colorability) of girth at least five drawn in any fixed surface.
For triangle-free graphs, the situation is more complicated, as there are infinitely many
minimal triangle-free non-$3$-colorable graphs which can be drawn in any surface other than the sphere.
Nevertheless, Dvořák, Král' and Thomas~\cite{trfree4} gave a rough structural characterization of $4$-critical triangle-free embedded graphs,
sufficient to design a linear-time algorithm to test $3$-colorability of triangle-free graphs drawn in a fixed surface~\cite{trfree7}.
A more detailed description of such $4$-critical graphs was given by Dvořák and Lidický~\cite{cylgen-part3}.

Before discussing further questions related to $3$-colorability of embedded triangle-free graphs, let us briefly introduce
the theory of hyperbolicity developed by Postle and Thomas~\cite{PosThoHyperb}.  A class $\GG$ of graphs embedded in closed surfaces (which possibly can have a boundary)
is \emph{hyperbolic} if there exists a constant $c_{\GG}$ such that for each graph $G\in\GG$ embedded in a surface $\Sigma$
and each open disk $\Lambda\subset \Sigma$ whose boundary $\partial \Lambda$ intersects $G$ only in vertices, the number of vertices of $G$ in $\Lambda$
is at most $c_{\GG}(|\partial \Lambda\cap G|-1)$.  The class is \emph{strongly hyperbolic} if the same holds for all sets $\Lambda\subset\Sigma$ homeomorphic to an open cylinder
(sphere with two holes).  The importance of these notions in the study of graph colorings stems from the following facts.
Firstly, the classes of minimal obstructions to many kinds of colorings form strongly hyperbolic families.  This is the case e.g. for
$6$-critical graphs (minimal obstructions to $5$-colorability)~\cite{pothom,lukethe} and for $4$-critical graphs of girth at least five~\cite{trfree3}.
Secondly, hyperbolicity or strong hyperbolicity is often relatively easy to establish, as it only involves
dealing with the planar subgraphs drawn in $\Lambda$.  Finally, hyperbolicity and strong hyperbolicity has a number of important
consequences, such as the following.  Recall the \emph{edge-width} of a graph drawn in a surface is the length of a shortest
non-contractible cycle of the graph.
\begin{theorem}[Postle and Thomas~\cite{PosThoHyperb}]
If $\GG$ is a hyperbolic class, then each graph in $\GG$ drawn in a surface without boundary of Euler genus $g$ has edge-width $O(\log g)$.
\end{theorem}
For example, since $4$-critical graphs of girth at least five are strongly hyperbolic~\cite{trfree3}, there exists a constant $c_g=O(\log g)$ such that every graph of girth at least five drawn in a surface of Euler genus
$g$ with edge-width at least $c_g$ is $3$-colorable.
\begin{theorem}[Dvořák and Kawarabayashi~\cite{dvokawalg}, Postle and Thomas~\cite{PosThoHyperb}]
For any hyperbolic class $\GG$, a surface $\Sigma$, and an integer $k$, there exists a linear-time algorithm to decide whether
a graph drawn in $\Sigma$ with at most $k$ vertices contained in the boundary of $\Sigma$ has a subgraph belonging to $\GG$.
\end{theorem}
For example, since $4$-critical graphs of girth at least five are strongly hyperbolic~\cite{trfree3},
it is possible to test in linear time whether a graph of girth at least five drawn in a fixed surface is $3$-colorable,
or more generally, whether a precoloring of a bounded number of vertices in such a graph extends to a $3$-coloring.
\begin{theorem}[Postle and Thomas~\cite{PosThoHyperb}]\label{thm-potsmall}
If $\GG$ is a strongly hyperbolic class, then each graph in $\GG$ drawn in a surface $\Sigma$ of Euler genus $g$ with boundary $\partial\Sigma$
has $O(g+|\partial \Sigma\cap V(G)|)$ vertices.
\end{theorem}
For example, since $4$-critical graphs of girth at least five are strongly hyperbolic~\cite{trfree3},
for every integer $g\ge 0$, there exists a constant $s_g=O(g)$ such that every non-$3$-colorable graph of girth at least five and Euler genus
$g$ contains a non-$3$-colorable subgraph with at most $s_g$ vertices.

With these results in mind, let us return to the discussion of other variants of graph coloring in the context of
triangle-free embedded graphs, especially the \emph{list coloring} and the \emph{fractional coloring}.
Let $L$ be an assignment of lists of colors to vertices of a graph $G$.  An \emph{$L$-coloring} of $G$ is
a proper coloring of $G$ such that the color of each vertex $v$ belongs to the list $L(v)$.  A graph $G$ is \emph{list $k$-colorable} if
it is $L$-colorable for every assignment $L$ of lists of size at least $k$.  Clearly, a list $k$-colorable
graph is also $k$-colorable.  The direct counterpart of Grötzsch' theorem~\cite{grotzsch1959} is false:
Voigt~\cite{voigt1995} found a triangle-free planar graph which is not list $3$-colorable.
On the other hand, Thomassen~\cite{thomassen1995-34} proved that planar graphs of girth at least five are list $3$-colorable.
Furthermore, Dvořák and Kawarabayashi~\cite{dk} proved that minimal obstructions of girth at least $5$
to list $3$-colorability are hyperbolic, and finally Postle~\cite{postle3crit} proved they are strongly hyperbolic.
Hence, all the aforementioned results for $3$-coloring of embedded graphs of girth at least five also
apply in the list coloring setting.

Let us now turn our attention to fractional coloring.
A function that assigns sets to all vertices of a graph is a \emph{set coloring} if the sets
assigned to adjacent vertices are disjoint.
For positive integers $a$ and $b$, a {\em $(b:a)$-coloring} of a graph $G$ is a set coloring
which to each vertex assigns an $a$-element subset of $\{1,\ldots, b\}$.
The concept of $(b:a)$-coloring is a generalization of the conventional vertex coloring. In fact,
a $(b:1)$-coloring is exactly an ordinary proper $b$-coloring.
The {\em fractional chromatic number} of $G$, denoted by $\chi_f(G)$, is the infimum of the
fractions $b/a$ such that $G$ admits a $(b:a)$-coloring. Note that $\chi_f(G)\leq \chi(G)$ for any graph $G$,
where $\chi(G)$ is the chromatic number of $G$.

Grötzsch' theorem can be improved only very mildly in the fractional coloring setting.
Jones~\cite{Jon84} found for each integer $n$ such that $n\equiv 2\pmod 3$
an $n$-vertex triangle-free planar graph with fractional chromatic number exactly
$3-\frac{3}{n+1}$.  On the other hand, Dvořák, Sereni and Volec~\cite{frpltr} proved that every triangle-free planar graph with $n$ vertices is
$(9n:3n+1)$-colorable, and thus its fractional chromatic number is at most $3-\frac{3}{3n+1}$.
The examples given by Jones~\cite{Jon84} have many $4$-cycles, leading Dvořák and Mnich~\cite{dmnich} to conjecture the following.
\begin{conjecture}\label{conj-g5}
There exists a constant $c<3$ such that every planar graph of girth at least five has fractional chromatic number
at most $c$.
\end{conjecture}
While the conjecture is open in general, in a followup paper we will prove that it holds for graphs with bounded maximum degree
(i.e., for any $\Delta$, there exists $c_\Delta<3$ such that planar graphs of girth at least five and maximum degree at most $\Delta$
have fractional chromatic number at most $c_\Delta$).  As a key step in that argument, we need to show that the class of minimal
obstructions of girth at least five to $(6:2)$-colorability is strongly hyperbolic.

Since every $3$-colorable graph is also $(6:2)$-colorable, it might seem obvious that this can be done by some modification
of the argument of Dvořák, Král' and Thomas~\cite{trfree3} for $3$-colorings.  Somewhat surprisingly, this turns out not to be the
case (the argument is based on reducible configurations, and even the simplest one---a $5$-cycle of vertices of degree three---does
not work for $(6:2)$-coloring).  Fortunately, the list-coloring argument of Postle~\cite{postle3crit} does the trick, subject to
extensive groundwork and some minor modifications.  As an added benefit, we obtain the result for the list variant of fractional
coloring.  For an assignment $L$ of lists to vertices of a graph $G$ and a positive integer $a$, a set coloring $\varphi$ of $G$
is an \emph{$(L:a)$-coloring} if $\varphi(v)$ is an $a$-element subset of $L(v)$ for every $v\in V(G)$.  Let $S$ be a proper subgraph
of $G$.  We say $G$ is \emph{$(a,L,S)$-critical} if for every proper subgraph $H$ of $G$ containing $S$, there exists
an $(L:a)$-coloring of $S$ which extends to an $(L:a)$-coloring of $H$, but not to an $(L:a)$-coloring of $G$.
In particular, denoting by $\varnothing$ the null subgraph (with no vertices and edges), $G$ is $(a,L,\varnothing)$-critical
if and only if $G$ is a minimal non-$(L:a)$-colorable graph.
Our main result is the following.
\begin{theorem}\label{thm-postrong}
Let $\GG$ be the class of graphs of girth at least five drawn in surfaces such that if $G\in \GG$ is drawn in a surface $\Sigma$
and $S$ is the subgraph of $G$ drawn in the boundary of $\Sigma$, then $G$ is $(a,L,S)$-critical for some
positive integer $a$ and an assignment $L$ of lists of size at least $3a$ to vertices of $G$.
Then $\GG$ is strongly hyperbolic.
\end{theorem}
For example, together with Theorem~\ref{thm-potsmall} this implies that for every integer $g\ge 0$, there exists a constant $s_g=O(g)$
such that if $G$ is a graph of girth at least five drawn in a surface of Euler genus $g$, and $G$ is not $(L:a)$-colorable
for some positive integer $a$ and an assignment $L$ of lists of size at least $3a$ to vertices of $G$, then
$G$ contains a subgraph with at most $s_g$ vertices which also is not $(L:a)$-colorable.
The main difficulty in proving Theorem~\ref{thm-postrong} is the need to establish the following result.
For a positive integer $a$ and an assignment $L$ of lists of size $2a$ or $3a$ to vertices of a graph $G$,
a \emph{flaw} is an edge joining two vertices with lists of size $2a$.

\begin{theorem}\label{thm-cyl}
Let $a$ be a positive integer.  Let $G$ be a plane graph of girth at least $5$ and let $f_1$ and $f_2$ be faces of $G$.  Let $L$ be a list assignment for $G$ such that
$|L(v)|\in\{2a,3a\}$ for all $v\in V(G)$ and all vertices with list of size $2a$ are incident with $f_1$ or $f_2$.
If each flaw is at distance at least three from any other vertex with list of size $2a$ and at least four from any other flaw, then $G$ is $(L:a)$-colorable.
\end{theorem}

A weaker form of Theorem~\ref{thm-cyl} for list $3$-coloring was proven by Thomassen~\cite{thomassen-surf};
in his formulation, the lists must be subsets of $\{1,2,3\}$, and more assumptions are made on distances between flaws and other vertices
with lists of size $2$.  Thomassen's argument can essentially be modified to work for list $(3a:a)$-colorings as well;
however, the argument is quite long and complicated and as presented in~\cite{thomassen-surf}, leaves quite a lot of details to the reader
to work out.  Hence, rather than trying to verify all the details in the fractional list coloring setting, we developed a simpler proof\footnote{Let us remark
our writeup is a bit longer than Thomassen's (17 pages including the auxiliary results compared to 11 pages), but we go into substantially greater detail in its
presentation.} along the same lines, which takes the majority of this paper.  First, in Section~\ref{sec-oneface}, we prove an auxiliary claim (Theorem~\ref{thm-2flaws}) regarding
graphs with a precolored path and at most two flaws.  Using this result, we can relatively easily deal with the most technical parts
of the proof of Theorem~\ref{thm-cyl}---the case where $f_1$ and $f_2$ are close to each other.  Theorem~\ref{thm-cyl} is established
in Section~\ref{sec-cyl}.  Finally, in Section~\ref{sec-hyper}, we discuss the modifications to the argument of Postle~\cite{postle3crit}
needed to prove Theorem~\ref{thm-postrong}.

\section{Graphs with two flaws}\label{sec-oneface}

Let $a$ be a positive integer, let $G$ be a plane graph and let $P$ be a path contained in the boundary of its outer face;
we view the path $P$ as directed, i.e., one of its endvertices is designated to be the first vertex of $P$.
A list assignment $L$ for $G$ is \emph{$(a,P)$-valid} if $|L(v)|=3a$ for every vertex $v\in V(G)$
not incident with the outer face of $G$, $|L(v)|\in\{2a,3a\}$ for every vertex $v\in V(G)\setminus V(P)$
incident with the outer face of $G$, and $P$ is $(L:a)$-colorable.
A \emph{flaw} (with respect to $P$) is an edge $uv$ of $G$ with $u,v\not\in V(P)$ such that $|L(u)|=|L(v)|=2a$.
The first vertex $p$ of $P$ is \emph{adjacent} to the flaw $uv$ if $pu\in E(G)$, and it is \emph{connected} to the flaw $uv$ if either $pu\in E(G)$, or $G$ contains
a path $pxyuv$ with $x,y\not\in V(P)$, $|L(x)|=2a$, and $|L(y)|=3a$.  If $p$ is connected to a flaw in a unique way and $P$ has length two, then
let $c(G,P,L)$ be the set defined as follows. If $pu\in E(G)$, then let $c(G,P,L)=L(u)$.
Otherwise, let $c'$ be an $a$-element subset of $L(y)\setminus L(u)$ and let $c(G,P,L)$ be an $a$-element
subset of $L(x)\setminus c'$.  If $P$ has length at most one or $p$ is not connected to a flaw, then let $c(G,P,L)=\emptyset$.
For a set $c$ of colors and a vertex $p$, an $(L:a)$-coloring is \emph{$(p,c)$-disjoint} if the color set of $p$ is disjoint from $c$.
For a cycle $C$ in $G$, let $\cin(C)$ denote the subgraph of $G$ drawn in the closed disk bounded by $C$, and let $\cex(C)$
denote the subgraph of $G$ drawn in the complement of the open disk bounded by $C$.

\begin{theorem}\label{thm-2flaws}
Let $a$ be a positive integer, let $G$ be a plane graph of girth at least $5$ and let $P=p_0\ldots p_\ell$ be a path of length at most two contained in the boundary of its outer face.
Let $L$ be an $(a,P)$-valid list assignment for $G$ with at most two flaws, such that distance between flaws is at least three.
Assume furthermore that if $\ell=2$, then either $p_0$ is not connected to a flaw or $p_0$ is connected to a flaw in a unique way.
Then every $(p_0,c(G,P,L))$-disjoint $(L:a)$-coloring of $P$ extends to an $(L:a)$-coloring of $G$.
\end{theorem}
\begin{proof}
Suppose for a contradiction that $G$ is a counterexample with $|V(G)|+|E(G)|$ minimum, and subject to that with $\ell$ maximum.
Let $c=c(G,P,L)$ and let $\psi_c$ denote a $(p_0,c)$-disjoint $(L:a)$-coloring of $P$ which does not extend to an $(L:a)$-coloring of $G$
(such a coloring exists since $G$ is a counterexample).

Note that $|L(v)|\le\deg(v)a$ for every $v\in V(G)\setminus V(P)$, as otherwise $|L(v)|\ge (\deg(v)+1)a$ by the assumption that $L$ is $(a,P)$-valid,
$\psi_c$ extends to an $(L:a)$-coloring of $G-v$ by the minimality of $G$, and we can greedily color $v$, obtaining an $(L:a)$-coloring
of $G$ which extends $\psi_c$, contradicting the choice of $\psi_c$.
In particular, all vertices not incident with the outer face have degree at least three.  Furthermore, $G$ is clearly connected.

Suppose that $G$ is not $2$-connected; then $G=G_1\cup G_2$ for proper induced subgraphs $G_1$ and $G_2$ intersecting in a single vertex $v$.
Without loss of generality, $P\not\subseteq G_2$.  Let $P_1=P\cap G_1$.  If $P_1\subseteq G_1$, then let $P_2=v$, otherwise let $P_2=P\cap G_2$.
By the minimality of $G$, the restriction of $\psi_c$ to $P_1$
extends to an $(L:a)$-coloring $\varphi_1$ of $G_1$.  Since $P_2$ has length at most one, the restriction of $\psi_c\cup \varphi_1$ to $P_2$ 
extends to an $(L_2:a)$-coloring $\varphi_2$ of $G_2$ by the minimality of $G$.
Then $\varphi_1\cup \varphi_2$ is an $(L:a)$-coloring of $G$ extending $\psi_c$, which is a contradiction.
Hence, $G$ is $2$-connected.

Let $K$ be the cycle bounding the outer face of $G$.  Analogously to the previous paragraph, we see that if $K$ has a chord, then $P$ has length two
and the chord is incident with $p_1$.  Suppose $p_1v$ is such a chord, and let $G=G_1\cup G_2$ for proper induced subgraphs $G_1$ and $G_2$ intersecting in $p_1v$,
where $p_0\in V(G_1)$ and $p_2\in V(G_2)$.
By the minimality of $G$, the restriction of $\psi_c$ to $p_1p_2$ extends to an $(L:a)$-coloring $\varphi_2$ of $G_2$.
Note that $c(G_1,p_0p_1v,L)\subseteq c(G,P,L)$, and thus 
the restriction of $\psi_c\cup \varphi_2$ to $p_0p_1v$ extends to an $(L:a)$-coloring $\varphi_1$ of $G_1$ by the minimality of $G$.
Then $\varphi_1\cup \varphi_2$ is an $(L:a)$-coloring of $G$ extending $\psi_c$, which is a contradiction.
Consequently, $K$ is an induced cycle.

\claim{cl-cpath}{Let $H$ be a proper subgraph of $G$ and let $Q=q_0\ldots q_k$ be an induced path in $H$ contained in the boundary
of the outer face of $H$, where $4\le k\le 8$.  Suppose that $P\cap H\subseteq Q$, every vertex $v\in V(H)\setminus V(Q)$ satisfies $|L(v)|=3a$,
and the distance between $q_0$ and $q_k$ in $H$ is at least $4$.  If no vertex of $H$ has more than two neighbors in $Q$,
then every $(L:a)$-coloring $\psi$ of $Q$ extends to an $(L:a)$-coloring of $H$.}
\begin{subproof}
Suppose first $k\le 6$.  Let $L'$ be a list assignment for $H-\{q_0,q_k\}$ such that
\begin{itemize}
\item $L'(q_1)=\psi(q_1)\cup \psi(q_2)$, $L'(q_{k-1})=\psi(q_{k-1})\cup \psi(q_{k-2})$,
\item for each vertex $v\not\in\{q_1,q_{k-1}\}$ with a neighbor $q\in\{q_0,q_k\}$, $L'(v)$ is a $2a$-element subset of $L(v)\setminus\psi(q)$,
and
\item $L'(z)=L(z)$ for any other vertex $z$ of $H-\{q_0,q_k\}$.
\end{itemize}
Note that the neighbors of $q_0$ and $q_k$
form an independent set, since $G$ has girth at least $5$ and the distance between $q_0$ and $q_k$
is at least $4$; and in particular $H-\{q_0,q_k\}$ with list assignment $L'$ has no flaws (with respect to the path $q_2\ldots q_{k-2}$).
By the minimality of $G$, we conclude that the restriction of $\psi$ to $q_2\ldots q_{k-2}$
extends to an $(L':a)$-coloring of $H-\{q_0,q_k\}$, which gives an $(L:a)$-coloring of $H$ extending $\psi$.

For $k\ge 7$, we prove the claim by induction on the number of vertices of $H$.
Suppose $H$ contains an induced cycle $C$ of length at most $8$ with $\cin(C)\neq C$.
By the induction hypothesis, $\psi$ extends to an $(L:a)$-coloring $\varphi_1$ of $\cex(C)$.
Let $uv$ be an edge of $E(C)\setminus E(P)$. Since $G$ has girth at least $5$, the distance between $u$ and $v$
in $\cin(C)-uv$ is at least $4$.  Since $G$ has girth at least $5$, no vertex has more than two neighbors in $C$.
By the induction hypothesis, the restriction of $\varphi_1$ to the path $C-uv$ extends to an $(L:a)$-coloring of $\cin(C)-uv$,
giving an $(L:a)$-coloring of $H$ extending $\psi$.
Hence, we can assume that $\cin(C)=C$ for every induced cycle of length at most $8$ in $H$.

Suppose some vertex $v\not\in V(Q)$ has two neighbors $q_i$ and $q_j$ in $Q$, with $i<j$.
Note that $j-i\le k-2$, since the distance between $q_0$ and $q_k$ is at least $4$.
Then $H=H_1\cup H_2$, where $H_1$ and $H_2$
are proper induced subgraphs of $H$ intersecting in $q_ivq_j$ and the outer face of $H_1$ is bounded by the cycle $C_1=q_i\ldots q_jv$.
Since $v$ has at most two neighbors in $Q$, the cycle $C_1$ is induced, and since its length is $j-i+2\le k\le 8$, we have
$H_1=\cin(C_1)=C_1$.  Let $P_2=q_0\ldots q_ivq_j\ldots q_k$.  Note that no vertex of $H_2$ can be adjacent to $v$ and two other vertices
of $P_2$, since the distance between $q_0$ and $q_k$ is at least $4\ge k-4$.  By the induction hypothesis, we can extend
the coloring of $P_2$ given by $\psi$ and by coloring $v$ by an $a$-element subset of $L(v)\setminus (\psi(q_i)\cup \psi(q_j))$
to an $(L:a)$-coloring of $H_2$.  This gives an $(L:a)$-coloring of $H$ extending $\psi$.  Hence, we can assume that no vertex
$v\not\in V(Q)$ has more than one neighbor in $Q$.

If $k=8$, then note that by planarity and the assumption that $G$ has girth at least $5$, $G$ cannot contain both a path of length
three between $q_0$ and $q_7$ and between $q_1$ and $q_8$.  By symmetry, we can assume $G$ does not contain such a path between
$q_0$ and $q_7$.  In case $k=7$, this is true as well by the assumption that the distance between $q_0$ and $q_k$ is at least $4$.
We consider the graph $H'=H-\{q_0,q_7,q_8\}$ with the list assignment $L'$ obtained as follows:
we set $L'(q_1)=\psi(q_1)\cup \psi(q_2)$, $L'(q_i)=\psi(q_i)\cup\psi(q_{i-1})$ for $i\in \{5,6\}$,
for each vertex $v\not\in V(Q)$ with a neighbor $x\in \{q_0,q_7,q_8\}$, we choose $L'(v)$ as a $2a$-element subset of $L(v)\setminus\psi(x)$,
and we let $L'(z)=L(z)$ for every other vertex $z$ of $H'$.
Since $G$ has girth at least $5$, since no vertex $v\in V(H)\setminus V(Q)$ has more than one neighbor in $Q$, and since
the distance between $q_0$ and $\{q_7,q_8\}$ is at least $4$, we conclude that $q_5q_6$ is the only flaw (with respect to $q_2q_3q_4$) in $H'$,
and that this flaw is not connected to $q_2$.  By the minimality of $G$, the restriction of $\psi$ to $q_2q_3q_4$
extends to an $(L':a)$-coloring of $H'$, which gives an extension of $\psi$ to an $(L:a)$-coloring of $H$.
\end{subproof}

A cycle $C$ in $G$ is \emph{tame} if $\cin(C)=C$, or $|C|\ge 8$ and $\cin(C)$ consists of $C$ and a chord of $C$, or $|C|\ge 9$ and
$\cin(C)$ consists of $C$ and a vertex with three neighbors in $C$.
\claim{cl-cyc}{All $(\le\!9)$-cycles in $G$ are tame.}
\begin{subproof}
Suppose that $G$ has a non-tame $(\le\!9)$-cycle, and choose such a cycle $C$ with $\cin(C)$ minimal.
Note that $C$ is an induced cycle in $\cin(C)$ and no vertex of $\cin(C)$ has three neighbors in $C$:
Otherwise, since $G$ has girth at least five and $C$ is not tame, there would exist a $(\le\!6)$-cycle $C'\neq C$ in $\cin(C)$
with $\cin(C')\neq C'$;  but then $C'$ would not be tame, contradicting the choice of $C$.

By the minimality of $G$, $\psi_c$ extends to an $(L:a)$-coloring $\varphi_1$ of $\cex(C)$.
Let $uv$ be any edge of $E(C)\setminus E(P)$.  Since $G$ has girth at least $5$, the distance between $u$ and $v$
in $\cin(C)-uv$ is at least $4$.  Hence, by \refclaim{cl-cpath}, the restriction of $\varphi_1$
to the path $C-uv$ extends to an $(L:a)$-coloring $\varphi_2$ of $\cin(C)-uv$.
However, then $\varphi_1\cup\varphi_2$ is an $(L:a)$-coloring of $G$ extending $\psi_c$,
which is a contradiction.  Consequently, $G$ cannot contain a non-tame $(\le\!9)$-cycle.
\end{subproof}

\claim{cl-35}{Suppose $C=q_1\ldots q_{11}$ is a cycle in $G$ disjoint from $V(P)$
and $q_1q_8,q_2q_6\in E(G)$, $|L(q_1)|=3a$, and $|L(q_i)|=\deg_G(q_i)a$ for $2\le i\le 11$.
Then $q_1$ has degree at least $5$ in $G$.}
\begin{subproof}
Suppose for a contradiction that $q_1$ has degree at most $4$.  By \refclaim{cl-cyc},
$H=C+\{q_1q_8,q_2q_6\}$ is an induced subgraph of $G$.  By the minimality of $G$,
$\psi_c$ extends to an $(L:a)$-coloring $\varphi$ of $G-V(H)$.  For $v\in V(H)$,
let $L'(v)$ be a subset of $L(v)$ disjoint from the color sets given to neighbors of $v$
in $V(G)\setminus V(H)$, such that $|L'(v)|=\deg_H(v)a$ for $v\in\{q_2,\ldots, q_{11}\}$
and $|L'(q_1)|=2a$.  Since $|L'(q_1)|=|L'(q_3)|=2a$ and $|L'(q_2)|=3a$, there exist
$a$-element sets $c_1\subset L'(q_1)$ and $c_3\subset L'(q_3)$ such that $|L'(q_2)\setminus (c_1\cup c_3)|\ge 2a$.
We color $q_1$ by $c_1$ and $q_3$ by $c_3$ and greedily extend this coloring to an $(L':a)$-coloring $\varphi_2$ of $H$.
However, then $\varphi_1\cup\varphi_2$ is an $(L:a)$-coloring of $G$ extending $\psi_c$,
which is a contradiction.
\end{subproof}

Recall $K$ is the cycle bounding the outer face of $G$.  If $G=K$, then since
$|L(v)|\le\deg(v)a$ for every $v\in V(G)\setminus V(P)$ and the distance between flaws is at least three, this is only possible if $\ell=2$ and $|K|=5$.
If $G$ consists of $K=x_1\ldots x_k$ and a vertex $z$ with three neighbors on $K$, then since the distance between flaws is at least three,
it follows that $k=9$, $\ell=2$, and if say $P=x_4x_5x_6$, then $z$ is adjacent to $x_2$, $x_5$, and $x_8$, and $|L(x_1)|=|L(x_3)|=|L(x_7)|=|L(x_9)|=2a$.
In either case, $\psi_c$ extends to an $(L:a)$-coloring of $G$ by the choice of $c=c(G,L,P)$, which is a contradiction.
Since $K$ is an induced cycle, it follows that $K$ is not tame, and by \refclaim{cl-cyc}, we have $|K|\ge 10$.

\claim{cl-2chor}{Let $Q=u_1u_2u_3$ be a path in $G$ with $u_1,u_3\in V(K)$ and $u_2\not\in V(K)$, such that neither $u_1$ nor $u_3$ is the middle vertex of $P$
when $P$ has length two.  Let $G=G_1\cup G_2$, where $G_1$ and $G_2$ are proper induced subgraphs of $G$ intersecting in $Q$
and $P\subset G_1$.  Then for $i\in \{1,3\}$, either $u_i\in V(P)$ or $|L(u_i)|=3a$, and $u_i$ is connected to a flaw (with respect to $Q$) in $G_2$ in a unique way.}
\begin{subproof}
By the minimality of $G$, $\psi_c$ extends to an $(L:a)$-coloring $\varphi_1$ of $G_1$.
Let $\psi$ be the restriction of $\varphi_1$ to $u_1u_2u_3$.  Since $\psi_c$ does not extend to an $(L:a)$-coloring of $G$,
we conclude that $\psi$ does not extend to an $(L:a)$-coloring of $G_2$.  By the minimality of $G$, it follows that
for $i\in \{1,3\}$, the vertex $u_i$ is connected to a flaw in $G_2$.  Since the distance between flaws in $G$ is at least three, this
excludes the case that $u_i\not\in V(P)$ and $|L(u_i)|=2a$.  In particular, every path $xyz$ in $G_2$ with $x,z\in V(K)$, $z\not\in V(P)$ and $|L(z)|=2a$
is a subpath of $K$.  Consequently, the path connecting $u_i$ to a flaw in $G_2$ in contained in $K$, and thus it is unique.
\end{subproof}

In particular, we have the following useful observation.
\claim{cl-2chorcor}{Let $Q$ be a path in $G$ of length at most two, whose ends belong to $K$ and are not equal to the middle vertex of $P$ when
$P$ has length two.  If $Q$ has length two, furthermore assume that $Q$ has an endvertex with list of size $2a$ not belonging to $V(P)$.
Then $Q$ is a subpath of $K$.}

\claim{cl-3chor}{Let $Q=u_1u_2u_3u_4$ be a path in $G$ with $u_1\in V(K)$, $u_4\in V(K)\setminus V(P)$, and $u_2,u_3\not\in V(K)$, such that $|L(u_4)|=2a$ and $u_1$ is not the middle vertex of $P$
when $P$ has length two.  Let $G=G_1\cup G_2$, where $G_1$ and $G_2$ are proper induced subgraphs of $G$ intersecting in $Q$
and $P\subset G_1$.  Then $u_1$ and $u_3$ are connected to flaws (with respect to $u_1u_2u_3$) in $G_2$ in a unique way, and any $(u_1,c(G_2,u_1u_2u_3,L))$-disjoint $(L:a)$-coloring of $u_1u_2u_3u_4$ 
extends to an $(L:a)$-coloring of $G_2$.}
\begin{subproof}
By the minimality of $G$, $\psi_c$ extends to an $(L:a)$-coloring $\varphi_1$ of $G_1$.
Let $\psi$ be the restriction of $\varphi_1$ to $u_1u_2u_3$ and let $L'$ be the list assignment obtained from $L$
by setting $L'(u_4)=\varphi_1(u_3)\cup\varphi_1(u_4)$.  Since $\psi_c$ does not extend to an $(L:a)$-coloring of $G$,
the precoloring $\psi$ of $u_1u_2u_3$ does not extend to an $(L':a)$-coloring of $G_2$.  By the minimality of $G$, it follows that
$u_1$ and $u_3$ are connected to flaws in $G_2$, and by \refclaim{cl-2chorcor}, the paths connecting them to flaws
are subpaths of $K+u_3u_4$; hence, they are unique.

Consider now a $(u_1,c(G_2,u_1u_2u_3,L))$-disjoint $(L:a)$-coloring $\theta$ of $u_1u_2u_3u_4$.
Let $L''$ be the list assignment obtained from $L$ by setting $L''(u_4)=\theta(u_3)\cup\theta(u_4)$.
Note that $c(G_2,u_1u_2u_3,L)=c(G_2,u_1u_2u_3,L'')$, and thus the restriction of $\theta$ to $u_1u_2u_3$
extends to an $(L'':a)$-coloring $\varphi_2$ of $G_2$ by the minimality of $G$.  Observe that $\varphi_2$ is an $(L:a)$-coloring
of $G_2$ which extends $\theta$.
\end{subproof}

Let $K=p_\ell\ldots p_0v_1v_2\ldots v_k$, and let $v_0=p_0$.
Suppose first that $p_0$ is adjacent to a flaw, necessarily $v_1v_2$ since $K$ is an induced cycle (in case $\ell=0$
and the flaw is $v_kv_{k-1}$, we relabel the vertices).
If $\ell\le 1$, then let $\psi$ be obtained from $\psi_c$ by choosing $\psi(v_1)$ as an $a$-element subset
of $L(v_1)\setminus\psi_c(p_0)$, and let $P'=p_\ell\ldots p_0v_1$.  Note that $v_1$ is not connected to a
flaw, since the distance between flaws is at least three.  Hence, $\psi$ (and thus also $\psi_c$) extends
to an $(L:a)$-coloring of $G$ (recall we chose a counterexample with $\ell$ maximum).
This is a contradiction, and thus $\ell=2$.  Then $\psi_c(p_0)\cap L(v_1)=\emptyset$, since $c=L(v_1)$, and $p_0$ is not connected
to a flaw in $G-p_0v_1$.  Hence, $\psi_c$ extends to an $(L:a)$-coloring of $G-p_0v_1$ by the minimality of $G$,
and the resulting coloring is also proper in $G$.  This is a contradiction, showing that $p_0$ is not adjacent
to a flaw.  That is, the minimum index $b\ge 1$ such that $|L(v_b)|=3a$ satisfies $b\in\{1,2\}$.

\claim{cl-b1}{$|L(v_{b+1})|=2a$.}
\begin{subproof}
Suppose for a contradiction that $|L(v_{b+1})|=3a$.  Consider the graph $G'=G-v_{b-1}v_b$.
Let $\psi=\psi_c$ if $b=1$ and let $\psi$ be an $(L:a)$-coloring extending $\psi_c$ to the path $p_\ell\ldots p_0v_1$
if $b=2$.  Let $L'$ be the list assignment for $G'$ obtained from $L$ by choosing $L'(v_b)$ as a $2a$-element subset of
$L(v_b)\setminus \psi(v_{b-1})$ and if $b=2$, additionally setting $L'(v_1)=\psi(v_1)\cup \psi(p_0)$.
Since $K$ is an induced cycle, $v_b$ has no neighbor with list of size $2a$ in $G'$.  Hence, the distance between flaws of $G'$ is at least three.
Furthermore, since $G$ has girth at least $5$, $v_b$ is not adjacent to $p_0$; together with \refclaim{cl-2chor}, this implies that $p_0$
is not connected to a flaw in $G'$.  By the minimality of $G$, we conclude that $\psi_c$ extends to an $(L':a)$-coloring
of $G'$.  This gives an $(L:a)$-coloring of $G$ extending $\psi_c$, which is a contradiction.  Therefore, $|L(v_{b+1})|=2a$.
\end{subproof}

Suppose now that $b=2$ and $|L(v_4)|=2a$, and thus $p_0$ is connected to the flaw $v_3v_4$.
If $\ell\le 1$, then let $\psi$ be obtained from $\psi_c$ by choosing $\psi(v_1)$ as an $a$-element subset
of $L(v_1)\setminus\psi_c(p_0)$, and let $P'=p_\ell\ldots p_0v_1$.  Note that $v_1$ is not connected to a
flaw, since it has no neighbor with list of size $2a$ not belonging to $P$.  Hence, $\psi$ (and thus also $\psi_c$) extends
to an $(L:a)$-coloring of $G$ (since we chose a counterexample with $\ell$ maximum). This is a contradiction,
and thus $\ell=2$.  Then by the choice of $c$, $\psi_c$ extends to an $(L:a)$-coloring $\psi'$ of $p_2p_1p_0v_1v_2$
such that $\psi'(v_2)\cap L(v_3)=\emptyset$.  Let $L'$ be the list assignment for $G-v_2$ obtained
from $L$ by setting $L'(v_1)=\psi'(v_1)\cup \psi'(p_0)$ and for each vertex $v\neq v_1$ adjacent to $v_2$,
choosing $L'(v)$ as a $2a$-element subset of $L(v)\setminus \psi'(v_2)$.  By \refclaim{cl-2chorcor} and the assumption that $G$ has girth
at least $5$, we conclude that $p_0$ is not connected to a flaw in $G-v_2$ with the list assignment $L'$,
and by the minimality of $G$, there exists an $(L':a)$-coloring of $G-v_2$ extending $\psi_c$.  However,
this implies that $G$ has an $(L:a)$-coloring extending $\psi_c$, which is a contradiction.
Since we already excluded the possibility that $p_0$ is adjacent to a flaw, together with
\refclaim{cl-2chorcor} this implies the following claim.
\claim{cl-p0nofl}{The vertex $p_0$ is not connected to a flaw.}

We now prove two auxiliary claims concerning special configurations appearing close to $P$.

\claim{cl-badchords}{Suppose $|L(v_{b+3})|=2a$ (and consequently $|L(v_{b+2})|=3a$) and $G$ contains a $5$-face bounded by
a cycle $v_bv_{b+1}v_{b+2}x_2x_0$.  Then $x_2$ does not have a neighbor in $P$ and $G$ does not contain a path $x_iyz$
with $i\in \{0,2\}$, $y\not\in\{v_b,v_{b+2}\}$, $z\in V(K)\setminus V(P)$, and $|L(z)|=2a$.}
\begin{subproof}
Claim~\ref{cl-cyc} implies $x_0,x_2\not\in V(K)$, since $K$ is an induced cycle and $v_b$ and $v_{b+2}$ have degree greater than two.
Since $x_0$ has degree at least three, \refclaim{cl-cyc} also implies that $x_2$ has no neighbor in $P$.

Suppose now that $G$ contains a path $x_iyz$ as described in the statement of the claim, for some $i\in\{0,2\}$.  Let $G=G_1\cup G_2$, where $G_1$ and $G_2$ are proper induced subgraphs of $G$,
$P\subset G_1$, and $G_1$ intersects $G_2$ only in the path $Q=v_{b+i}x_iy$ or $Q=v_{b+i}x_iyz$, depending on whether $y\in V(K)$ or not.
Let $c'=c(G_2,v_{b+i}x_iy,L)$.

If $i=0$, then note that $v_b$ is not adjacent to a flaw in $G_2$, and thus $|c'|\le a$.
Hence, $\psi_c$ extends to a $(v_b,c')$-disjoint $(L:a)$-coloring $\psi$ of the path $p_\ell\ldots p_0v_1\ldots v_b$.
By the minimality of $G$ and \refclaim{cl-3chor}, any $(v_b,c')$-disjoint $(L:a)$-coloring of $Q$ extends to an $(L:a)$-coloring of $G_2$.
Since $\psi_c$ does not extend to an $(L:a)$-coloring of $G$, it follows that $\psi$ does not extend to an $(L:a)$-coloring of $G_1$.
Consider the graph $G_1-v_b$ with a list assignment $L_1$ obtained from $L$ by setting $L_1(v_1)=\psi(v_1)\cup \psi(p_0)$ if $b=2$,
and by choosing $L_1(v)$ as a $2a$-element subset of $L(v)\setminus \psi(v_b)$ for each neighbor $v$ of $v_b$ other than $v_{b-1}$.
By \refclaim{cl-2chorcor}, we conclude that such a neighbor $v$ cannot be adjacent to another vertex with list of size two,
and that $p_0$ is not connected to a flaw in $G_1-v_b$ with the list assignment $L_1$.  By the minimality of $G$,
$\psi_c$ extends to an $(L_1:a)$-coloring of $G_1-v_b$; however, this also implies that $\psi$ extends to an $(L:a)$-coloring of $G$, which is a contradiction.

Hence, $i=2$.  Observe $\psi_c$ extends to an $(L:a)$-coloring $\psi$ of the path $p_\ell\ldots p_0v_1\ldots v_{b+2}$ such that $\psi(v_{b+2})\cap L(v_{b+3})=\emptyset$.
By the minimality of $G$ and \refclaim{cl-3chor}, any $(v_b,c')$-disjoint $(L:a)$-coloring of $Q$ extends to an $(L:a)$-coloring of $G_2$,
and thus $\psi$ does not extend to an $(L:a)$-coloring of $G_1$.  Suppose first that $\ell\le 1$, or $b=1$, or $\ell=2$ and $p_2$ is not connected to a flaw in $G_1$.
Let $L'_1$ be a list assignment for $G_1'=G_1-\{v_{b+1}, v_{b+2}\}$ obtained by setting $L'_1(v_i)=\psi(v_{i-1})\cup \psi(v_i)$ for $i=1,\ldots, b$
and choosing $L'_1(x_2)$ as a $2a$-element subset of $L(x_2)\setminus \psi(v_{b+2})$.  By \refclaim{cl-2chor}, $x_2$ is not incident with a flaw in $G'_1$,
and the only vertices with list of size $2a$ (not belonging to $P$) at distance at most two from $v_b$ are $x_2$ and possibly $v_{b-1}$.
We created at most one flaw ($v_1v_2$ when $b=2$), and a flaw of $G$ belongs to $G_2$ by \refclaim{cl-2chor} and \refclaim{cl-3chor}; hence, $G'_1$ has at most two flaws, and the distance between them
is at least three.  If $\ell=2$, then by the assumptions either $b=1$ (and then $p_0$ is not connected to a flaw in $G'_1$), or
$b=2$ and $p_2$ is not connected to a flaw in $G$ (and then $p_2$ is not connected to a flaw in $G'_1$, either, since $v_1$, $v_2$, and $x_2$ are not adjacent to $p_2$).
By the minimality of $G$, we conclude that $\psi_c$ extends to an $(L'_1:a)$-coloring of $G'_1$.  Consequently, $\psi$ extends to an $(L:a)$-coloring of $G_1$,
which is a contradiction.

Hence, we can assume that $\ell=2$, $b=2$, and $p_2$ is connected to a flaw $uv$ in $G_1$.  Suppose now that the distance between $\{u,v\}$ and $\{x_0,x_2\}$
is at least three.  Then let $G''_1=G_1-\{v_2,v_3,v_4\}$ with list assignment $L''_1$ obtained from $L$ by setting $L''_1(v_1)=\psi(p_0)\cup \psi(v_1)$
and by choosing $L''_1(w)$ as a $2a$-element subset of $L(w)\setminus \psi(r)$ for each vertex $w\neq v_1$ with a neighbor $r\in \{v_2,v_4\}$.
By \refclaim{cl-2chorcor}, neither $x_0$ nor $x_2$ has a neighbor with list of size $2a$ not in $V(P)\cup\{x_0,x_2\}$.
Since $G$ has girth at least $5$ and by \refclaim{cl-cyc}, $v_1$ is at distance at least three in $G''_1$ from the newly created flaw $x_0x_2$,
and thus $p_0$ is not connected to a flaw in $G''_1$.  Note that $uv$ and $x_0x_2$ are the only flaws in $G_1$, since $G_2$ contains a flaw and $G$ contains at most two flaws.
By the assumption, the distance between the flaws $uv$ and $x_0x_2$ is at least three.  By the minimality of $G$, we conclude that $\psi_c$ extends to an $(L''_1:a)$-coloring of
$G''_1$.  Consequently, $\psi$ extends to an $(L:a)$-coloring of $G_1$, which is a contradiction.

Hence, by \refclaim{cl-2chorcor}, the distance between $\{x_0,x_2\}$ and $\{u,v\}$ is exactly two, and we can without loss of generality
assume $z=v$.  Since $y$ is connected to a flaw in $G_2$ by \refclaim{cl-2chor} and \refclaim{cl-3chor} and the distance between flaws in $G$ is at least three,
we conclude that $y\in V(K)$ and $y$ is connected but not adjacent to a flaw in $G_2$.  Since $G$ contains only two flaws,
$v_4$ is by \refclaim{cl-2chor} connected to the same flaw in $G_2$.  By \refclaim{cl-cyc}, we conclude that
$G_2$ consists of the $9$-cycle $v_4\ldots v_{11}x_2$ (where $y=v_{11}$) and a vertex $t$ adjacent to $x_2$, $v_6$, and $v_9$.

The flaw $uv=v_{13}v_{12}$ is connected to $p_2$; hence, the outer face of $G_1$ has length $11$ or $13$.  
Note that there exists a $(y,c(G_2,yx_2v_4,L))$-disjoint $(L:a)$-coloring $\psi_3$ of the path $P_3=v_{11}v_{12}\ldots p_2p_1p_0v_1$ of length $6$ or $8$.
If $G_3=G_1-\{v_3,v_4\}$ had an $(L:a)$-coloring extending $\psi_3$, then we could extend it to $G_1$ greedily and then to $G_2$ by the minimality of $G$ since $\psi_3$ is $(y,c(G_2,yx_2v_4,L))$-disjoint,
obtaining an $(L:a)$-coloring of $G$.  This is not possible, and thus $\psi_3$ does not extend to an $(L:a)$-coloring of $G_3$.  Furthermore, no vertex in $V(G_3)\setminus V(P_3)$
has list of size $2a$, and by \refclaim{cl-3chor}, the distance between $v_1$ and $v_{11}$ in $G_3$ is at least $4$.  
Since $\psi_3$ does not extend to an $(L:a)$-coloring of $G_3$, \refclaim{cl-cpath} together with \refclaim{cl-2chorcor} imply that $P_3$ has length $8$
and $G$ contains a vertex $w$ adjacent to $p_1$, $v_{11}$, and $v_{14}$.  Since $x_0$ has degree at least three,
\refclaim{cl-cyc} implies that $x_0$ is adjacent to $p_1$, and by \refclaim{cl-cyc} and \refclaim{cl-2chorcor}, we conclude this uniquely determines the whole graph $G$.
However, this contradicts \refclaim{cl-35}, applied to the cycle $x_2tv_9v_8\ldots v_2x_0$.

This contradiction shows that no path $v_{b+i}x_iyz$ as described in the statement of the claim exists.
\end{subproof}

\claim{cl-flawchords}{Suppose $b=1$, $|L(v_3)|=2a$, $|L(v_4)|=3a$, $|L(v_5)|=2a$, and a cycle $v_2v_3v_4x_4x_2$ (where possibly $x_2=v_1$) bounds a $5$-face.
Then $x_4$ does not have a neighbor in $P$ and $G$ does not contain a path $x_iyz$
with $i\in \{2,4\}$, $y\not\in\{v_2,v_4\}$, $z\in V(K)\setminus V(P)$, and $|L(z)|=2a$.}
\begin{subproof}
Note that $|L(v_2)|=2a$ by \refclaim{cl-b1}, and thus $v_2v_3$ is a flaw.

Suppose $x_4$ has a neighbor in $P$.  By \refclaim{cl-cyc} and the assumption that $G$ has girth at least $5$, we conclude that
$\ell=2$, $x_4p_2\in E(G)$, and $x_2=v_1$.  Let $G_2=G-\{p_1,p_0,v_1,v_2,v_3\}$.  By \refclaim{cl-2chor},
both $p_2$ and $v_4$ are connected to a flaw in $G_2$.  Since $G$ has at most two flaws and $v_2v_3$ is one of them,
both $p_2$ and $v_4$ are connected to the same flaw, and thus the outer face of $G_2$ is bounded by a $7$- or $9$-cycle.
By \refclaim{cl-cyc} and the fact that vertices with list of size $3a$ not in $P$ must have degree at least three, we conclude
that $G_2$ is bounded by a $9$-cycle and contains a vertex $w$ adjacent to $v_6$, $v_9$, and $x_4$.
However, this contradicts \refclaim{cl-35} for the cycle $x_4wv_9v_8\ldots v_1$.  Hence, $x_4$ has no neighbor in $P$.

Suppose now that $G$ contains a path $x_iyz$ as described in the statement of the claim, for some $i\in\{2,4\}$.  Let $G=G_1\cup G_2$, where $G_1$ and $G_2$ are proper induced subgraphs of $G$,
$P\subset G_1$, and $G_1$ intersects $G_2$ only in the path $Q=v_ix_iy$ or $Q=v_ix_iyz$, depending on whether $y\in V(K)$ or not.
If $i=2$, then by \refclaim{cl-2chorcor} we have $x_2\neq v_1$.  By \refclaim{cl-2chor} and \refclaim{cl-3chor}, $v_2$ is connected to a flaw in $G_2$.
However, since $v_2v_3$ is a flaw in $G$ (but not $G_2$, with respect to $Q$), this contradicts the assumption that the distance between flaws is at least three.

Hence, we have $i=4$.  By \refclaim{cl-2chor} and \refclaim{cl-3chor}, $G_2$ contains a flaw, and since $G$ has at most two flaws,
$v_2v_3$ is the only flaw in $G_1$.  Let $c'=c(G_2,v_4x_4y,L)$; the flaw of $G_2$ is at distance at least three from $v_2v_3$,
and thus it is not adjacent to $v_4$, implying $|c'|=a$.  Observe that there exists a $(v_4,c')$-disjoint $(L:a)$-coloring $\psi$ of
$p_\ell\ldots p_0v_1v_2v_3v_4$ extending $\psi_c$.  By \refclaim{cl-2chor} and \refclaim{cl-3chor}, $\psi$ does not extend to an $(L:a)$-coloring of $G_1$,
as otherwise the resulting coloring would further extend to $G_2$ and give an $(L:a)$-coloring of $G$ extending $\psi_c$.
If $x_2\neq v_1$, then let $G'_1=G_1-\{v_3,v_4\}$, otherwise let $G'_1=G_1-\{v_2,v_3,v_4\}$.  Let $L'_1$ be the list assignment for $G'_1$ obtained 
from $L$ by setting $L_1'(v_1)=\psi(v_1)\cup \psi(p_0)$, $L_1'(v_2)=\psi(v_2)\cup \psi(v_1)$ when $x_2\neq v_1$, and choosing $L_1'(x_4)$ as a $2a$-element
subset of $L(x_4)\setminus\psi(v_4)$.  Let $e=v_1v_2$ if $x_2\neq v_1$ and $e=v_1x_4$ otherwise.  Then $e$ is the only flaw in $G'_1$
and when $\ell=2$, the vertex $p_2$ is not connected to $e$, by \refclaim{cl-2chorcor} and the fact that $x_4$ is not adjacent to $p_2$.
Hence, $\psi_c$ extends to an $(L'_1:a)$-coloring of $G'_1$ by the minimality of $G$, and thus $\psi$ extends to an $(L:a)$-coloring of $G_1$,
which is a contradiction.
\end{subproof}

Recall that $L(v_i)=2a$ for $1\le i\le b-1$, $|L(v_b)|=3a$, and $|L(v_{b+1})|=2a$ by \refclaim{cl-b1}. 
Furthermore, by \refclaim{cl-p0nofl} the vertex $p_0$ is not connected to a flaw, and thus if $b=2$, then $|L(v_{b+2})|=3a$.  Let us now define a set $X$
of vertices of $G$ depending on the sizes of lists of $v_{b+2}$, $v_{b+3}$, and $v_{b+4}$ as follows.
\begin{itemize}
\item[(X1)] If $b=1$, $|L(v_3)|=2a$ and $|L(v_4)|=|L(v_5)|=3a$, then $X=\{v_2,v_3\}$.
\item[(X2)] If $b=1$, $|L(v_3)|=2a$, $|L(v_4)|=3a$, $|L(v_5)|=2a$, and $v_2v_3v_4$ is not a subpath of the boundary of a $5$-face, then $X=\{v_2,v_3,v_4\}$.
\item[(X2a)] If $b=1$, $|L(v_3)|=2a$, $|L(v_4)|=3a$, $|L(v_5)|=2a$, and a cycle $v_2v_3v_4rs$ (where possibly $s=v_1$) bounds a $5$-face, then $X=\{v_2,v_3,v_4, r,s\}$.
\item[(X3)] If $|L(v_{b+2})|=|L(v_{b+3})|=3a$, then $X=\{v_b,v_{b+1}\}$.
\item[(X4)] If $|L(v_{b+2})|=3a$, $|L(v_{b+3})|=2a$, and $v_bv_{b+1}v_{b+2}$ is not a subpath of the boundary of a $5$-face, then $X=\{v_b,v_{b+1},v_{b+2}\}$.
\item[(X4a)] If $|L(v_{b+2})|=3a$, $|L(v_{b+3})|=2a$, and a cycle $v_bv_{b+1}v_{b+2}rs$ bounds a $5$-face, then $X=\{v_b,v_{b+1},v_{b+2},r,s\}$.
\end{itemize}
Let $p$ be the minimum index such that $v_p\in X$, and let $m$ be the maximum index such that $v_m\in X$; note that $m\le 4$.
Observe that there exists an $(L:a)$-coloring $\psi$ of $G[V(P)\cup X\cup \{v_1,\ldots, v_{p-1}\}]$ extending $\psi_c$ such that
if $|L(v_{m+1})|=2a$, then $\psi(v_m)\cap L(v_{m+1})=\emptyset$; in cases (X2a) and (X4a), this holds since $r$ has no neighbor
in $P$ by \refclaim{cl-badchords} and \refclaim{cl-flawchords}. 

Note that since $G$ has girth at least $5$, each vertex of $V(G)\setminus X$ has at most
one neighbor in $X$. 
Furthermore, there is no vertex $v\in V(G)\setminus (X\cup V(P)\cup \{v_{p-1},v_{m+1}\})$ such that $|L(v)|=2a$ and $v$ has a neighbor $x\in X$,
since otherwise $Q=vx$ if $x\in V(K)$, $Q=vrv_m$ if $x=r$, or $Q=vsv_{m-2}$ if $x=s$ would contradict \refclaim{cl-2chorcor}.

Let $G'=G-X$ and let $L'$ be the list assignment defined as follows.  If $v\in V(G)\setminus (X\cup V(P)\cup \{v_1,\ldots,v_{p-1}\})$ has a neighbor $x\in X$,
then let $L'(v)$ be a $2a$-element subset of $L(v)\setminus\psi(x)$, otherwise let $L'(v)=L(v)$.  For $i=1,\ldots, p-1$, let $L'(v_i)=\psi(v_{i-1})\cup \psi(v_i)$.
The choice of $X$ and $\psi$ ensures that $|L'(v)|\in \{2a,3a\}$ for all $v\in V(G')\setminus V(P)$,
and thus $L'$ is an $(a,P)$-valid list assignment for $G'$.
Note that any $(L':a)$-coloring of $G'$ extending $\psi_c$ would combine with $\psi$ to an $(L:a)$-coloring of $G$, and thus $\psi_c$ does not extend to an $(L':a)$-coloring of $G'$.

Suppose $uv$ is a flaw of $G'$ which does not appear in $G$, say $|L(u)|=3a$.  If $|L(v)|=3a$, then both $u$ and $v$ would have a neighbor in $X$;
since $G$ has girth at least $5$ and satisfies \refclaim{cl-cyc} (and noting the assumption of the non-existence of a $5$-cycle
containing the path $v_{m-2}v_{m-1}v_m$ in cases (X2) and (X4)), this is not possible.  The case $|L(v)|=2a$ is excluded by \refclaim{cl-2chorcor},
\refclaim{cl-badchords}, and \refclaim{cl-flawchords}.  Consequently, $G'$ has no new flaws, and in particular it has at most two flaws and the distance between them
is at least three.  Furthermore, $p_0$ is not adjacent to a flaw for the same reason.  By the minimality of $G$, since $\psi_c$ does not extend to an $(L':a)$-coloring of $G'$,
we conclude that $\ell=2$ and $p_0$ is connected but not adjacent to a flaw.  Let $Q=p_0xyuv$ be a path in $G'$ such that $x,y,u,v\not\in V(P)$ and $|L'(x)|=|L'(u)|=|L'(v)|=2a$.
As we argued before, $uv$ is also a flaw in $G$; in particular, $|L(u)|=2a$ and $u\in V(K)$.  Clearly $p_0xyu$ is not a subpath of $K$, and by \refclaim{cl-2chorcor}, we conclude that
$x\not\in V(K)$.  But then \refclaim{cl-2chor} or \refclaim{cl-3chor} applied to $p_0xyu$ would imply that $p_0$ is connected to
a flaw in $G$, which is a contradiction.

Therefore, no counterexample to Theorem~\ref{thm-2flaws} exists.
\end{proof}

Let us remark that it would be possible to exactly determine the minimal graphs $G$
satisfying the assumptions of Theorem~\ref{thm-2flaws} such that some
$(L:a)$-coloring of $P$ does not extend to an $(L:a)$-coloring of $G$ (there are only finitely many).
However, the weaker statement we gave is a bit easier to prove and sufficient in the application
we have in mind.

\section{Graphs with two special faces}\label{sec-cyl}

Let us note an easy consequence of a special case of Theorem~\ref{thm-2flaws} with no precolored path, obtained by an argument taken from~\cite{thomassen-surf}.
\begin{corollary}\label{cor-distflaws}
Let $a$ be a positive integer.  Let $G$ be a plane graph of girth at least $5$ and let $f$ be a face of $G$.  Let $L$ be a list assignment for $G$ such that
$|L(v)|\in\{2a,3a\}$ for all $v\in V(G)$ and the vertices with list of size $2a$ are all incident with $f$.
If each flaw is at distance at least three from any other vertex with list of size $2a$ and at least four from any other flaw, then $G$ is $(L:a)$-colorable.
\end{corollary}
\begin{proof}
We prove the claim by the induction on $|V(G)|$.  If $G$ contains a flaw $xy$, then choose an $(L:a)$-coloring $\psi$ of $G[\{x,y\}]$ arbitrarily,
and let $L'$ be the list assignment for $G'=G-\{x,y\}$ obtained from $L$ by choosing $L'(v)$ as a $2a$-element subset of $L(v)\setminus \psi(z)$
for every vertex $v$ with a neighbor $z\in\{x,y\}$.  Note that the neighbors of $\{x,y\}$ form an independent set, are not adjacent to other
vertices with list of size $2a$, and their distance to other flaws is at least three, and thus $G'$ with the list assignment $L'$ satisfies the assumptions
of Corollary~\ref{cor-distflaws}.  By the induction hypothesis, there exists an $(L':a)$-coloring of $G'$, which together with $\psi$ forms an $(L:a)$-coloring
of $G$.

Hence, we can assume that there are no flaws.  Let $p$ be an arbitrary vertex incident with $f$.  By Theorem~\ref{thm-2flaws}, any $(L:a)$-coloring of $p$
extends to an $(L:a)$-coloring of $G$.
\end{proof}

We now strengthen Corollary~\ref{cor-distflaws} to the case vertices with lists of size $2a$ are incident with two faces of $G$.

\begin{proof}[Proof of Theorem~\ref{thm-cyl}]
Suppose for a contradiction that $G$ is a counterexample with the smallest number of vertices.  Without loss of generality, we can assume $f_1$ is the outer face of $G$.
Clearly $G$ is connected, has minimum degree at least $2$, and vertices with list
of size $3a$ have degree at least three.  Note that $G$ has no flaws: otherwise, $(L:a)$-color the vertices of a flaw $xy$, delete $x$ and $y$, remove the colors of $x$ and $y$
from the lists of their neighbors, and reduce further the lists of these neighbors to size $2a$.  The neighbors form an independent set, are not incident to other
vertices with list of size $2a$, and their distance to other flaws is at least three, implying that the resulting graph would be a counterexample smaller than $G$
(or contradict Corollary~\ref{cor-distflaws}, if both $f_1$ and $f_2$ are incident with $\{x,y\}$, so
that all vertices with list of size $2a$ in $G-\{x,y\}$ are incident with one face).

Furthermore, the faces $f_1$ and $f_2$ are bounded by cycles: otherwise, if say $f_1$ is not bounded by a cycle, then $G=G_1\cup G_2$ for proper
induced subgraphs intersecting in a vertex $p$ incident with $f_1$, such that $f_2$ is a face of $G_1$ and the boundary of the outer face of $G_2$ is part of the boundary of $f_1$.
By the minimality of $G$, the graph $G_1$ is $(L:a)$-colorable, and by Theorem~\ref{thm-2flaws}, the corresponding coloring of $p$ extends to an $(L:a)$-coloring of $G_2$,
together giving an $(L:a)$-coloring of $G$.

Let $C_1$ and $C_2$ be the cycles bounding $f_1$ and $f_2$, respectively.  Analogously to the proof of \refclaim{cl-cyc}, we conclude that
the following holds.
\claim{cl-cocyc}{Every $(\le\!9)$-cycle $C$ in $G$ such that $f_2$ is not contained in $\cin(C)$ is tame.}

Next, let us restrict short paths with both ends in $C_1$ or both ends in $C_2$.

\claim{cl-sp}{Let $Q=q_0\ldots q_k$ be a path of length $k\le 4$ in $G$.   If $k=4$, furthermore assume that $|L(q_4)|=2a$.
If $q_0,q_k\in V(C_i)$ for some $i\in\{1,2\}$, then
\begin{itemize}
\item $Q\subset C_i$, or
\item $k\ge 3$, $q_3\in V(C_i)$, $|L(q_0)|=|L(q_3)|=3a$, and $q_0$ and $q_3$ have a common neighbor in $C_i$ with list of size $2a$, or
\item $k=4$, $|L(q_0)|=3a$, and $q_0$ is adjacent to $q_4$ in $C_i$.
\end{itemize}
In the last two cases, the cycle consisting of $q_0q_1q_2q_3$ and the common neighbor of $q_0$ and $q_3$ bounds a $5$-face.}
\begin{subproof}
We prove the claim by induction on $k$, the case $k=0$ being trivial.
By symmetry, we can assume that $i=1$.  Furthermore, we can assume that $q_1, \ldots, q_{k-1}\not\in V(C_1)$, since otherwise the claim
follows by the induction hypothesis applied to subpaths of $Q$ between vertices belonging to $C_1$.  Let $G=G_1\cup G_2$, where $G_1$ and $G_2$ are proper induced subgraphs of $G$ intersecting
in $Q$ and $f_2$ is a face of $G_1$.   By the minimality of $G$, there exists an $(L:a)$-coloring $\varphi_1$ of $G_1$.
If $k\le 2$, then the restriction of $\varphi_1$ to $Q$ extends to an $(L:a)$-coloring of $G_2$ by Theorem~\ref{thm-2flaws}.
This gives an $(L:a)$-coloring of $G$, which is a contradiction.

If $k=3$, then let $L'$ be the list assignment for $G_2$ such that $L'(q_3)=\varphi_1(q_3)\cup \varphi_1(q_2)$
and $L'(x)=L(x)$ for any other vertex $x\in V(G_2)\setminus\{q_3\}$,
and let $\psi$ be the restriction of $\varphi_1$ to $Q-q_3$.   Note that all neighbors of $q_3$ in $C_2\cap G_2$ (if any) belong to $Q$,
and using the induction hypothesis (the case $k=1$), observe that $q_3$ has exactly one neighbor in $C_1\cap G_2$;
hence, $q_3$ has at most one neighbor $u\in V(G_2)\setminus V(Q)$ with list of size $2a$, and thus the list assignment $L'$ has at most one flaw $q_3u$
with respect to $Q-q_3$.  Since $G$ is not $(L:a)$-colorable, $\psi$ does not extend to an $(L':a)$-coloring of $G_2$, and thus
Theorem~\ref{thm-2flaws} implies the flaw $q_3u$ is connected to $q_0$.  By \refclaim{cl-cocyc}
and the fact that vertices with list of size $3a$ have degree at least three, this is only possible when $u$ is adjacent to $q_0$,
and thus $q_0$ and $q_3$ have a common neighbor $u$ with list of size $2a$.  The path $q_0uq_3$ is a subpath of $C_1$ by the induction hypothesis (the case $k=2$).

Suppose now that $k=4$.  Let $L'$ be the list assignment for $G_2$ such that $L'(q_4)=\varphi_1(q_4)\cup \varphi_1(q_3)$,
$L'(q_0)=\varphi_1(q_0)\cup \varphi_1(q_1)$, and $L'(x)=L(x)$ for any other vertex $x\in V(G_2)\setminus\{q_0,q_4\}$,
and let $\psi$ be the restriction of $\varphi_1$ to $Q-\{q_0,q_4\}$.  Again, using the induction hypothesis (the case $k=1$),
we conclude $G_2$ with the list assignment $L'$ has at most one flaw $q_0u$ with respect to $Q-\{q_0,q_4\}$.
Since $G$ is not $(L:a)$-colorable, Theorem~\ref{thm-2flaws} implies the flaw $q_0u$ is connected to $q_3$.  By \refclaim{cl-cocyc}
and the fact that vertices with list of size $3a$ have degree at least three, this is only possible when $u$ is adjacent to $q_3$.
If $u\neq q_4$, then the induction hypothesis (the case $k=2$) applied to the path $q_4q_3u$ would imply $q_3\in V(C_1)$, a contradiction.
Consequently, we conclude that $u=q_4$.  Hence, $q_0$ is adjacent to $q_4$, and by the induction hypothesis (the case $k=1$), $q_0q_4$ is an edge of $C_1$.

In the last two cases, the cycle consisting of $q_0q_1q_2q_3$ and the common neighbor $u$ of $q_0$ and $q_3$ bounds a $5$-face by \refclaim{cl-cocyc}.
Furthermore, since $|L(u)|=2a$ and $G$ has no flaws, it follows that $|L(q_0)|=|L(q_3)|=3a$.
\end{subproof}
In particular, note this implies $C_1$ and $C_2$ are induced cycles.  Furthermore, if $k=3$ and $|L(q_0)|=2a$ or $|L(q_3)|=2a$,
or if $k=4$ and $|L(q_0)|=|L(q_4)|=2a$, then $Q$ must be a subpath of $C_1$ or $C_2$.

We now show that $G$ cannot have a short path between $C_1$ and $C_2$, successively showing better and better bounds on the distance
between $C_1$ and $C_2$.  Suppose $Q$ is an induced path with ends in $C_1$ and $C_2$ and otherwise disjoint from $C_1\cup C_2$.
Each edge $e\not\in E(Q)$ incident with a vertex of $Q$ goes either to the left or to the right side of $Q$ (as seen from $f_2$).
Let $R(Q)$ denote the set of edges $e\not\in E(Q)$ incident with vertices of $Q$ from the right side.

\claim{cl-dist1}{The cycles $C_1$ and $C_2$ are disjoint.}
\begin{subproof}
Suppose for a contradiction that there exists a vertex $w\in V(C_1\cap C_2)$.
Let $C_1=wv_1v_2\ldots$ and $C_2=wu_1u_2\ldots$, labeled in the clockwise order.   Without loss of generality (by choosing $w$ as an endvertex of a path forming a connected component
of $C_1\cap C_2$), we can assume $u_1\neq v_1$, and since $C_1$ and $C_2$ are induced cycles, it follows that $v_1\not\in V(C_2)$ and $u_1\not\in V(C_1)$.
By symmetry, we can assume that $|L(u_1)|\le |L(v_1)|$.  

If $|L(v_1)|=3a$, then let $G'=G-R(w)$, let $\psi$ be an $(L:a)$-coloring of $w$ such that if $|L(u_1)|=2a$ then $\psi(w)\cap L(u_1)=\emptyset$,
and let $L'$ be the list assignment obtained from $L$ by choosing $L'(x)$ as a $2a$-element subset of $L(x)\setminus\psi(w)$ for every vertex $x\neq w$ incident with an edge in $R(w)$;
since $G$ has girth at least five, these vertices form an independent set.  Consider such a vertex $x$.  If $x\not\in V(C_1\cup C_2)$, then $x$ cannot have a neighbor $z\in V(C_1\cup C_2)$
distinct from $w$ by \refclaim{cl-sp} applied to the path $wxz$, and thus $x$ is not contained in a flaw of $G'$.  If $x\in V(C_1\cup C_2)$, then since $C_1$ and $C_2$ are induced cycles,
we have $x\in \{u_1,v_1\}$.  Note that $v_1$ has no neighbor $r\in V(C_2)$ other than $w$, by \refclaim{cl-sp} applied to the path $wv_1r$ and using the fact that $v_1\not\in V(C_2)$.
Since $C_1$ is an induced cycle, it follows that $v_1$ can only be contained in a flaw $v_1v_2$ in $G'$.  Analogously, $u_1$ can only be contained in a flaw $u_1u_2$ (and in this case $|L(u_1)|=3a$).
If $|L(u_1)|=3a$, then observe that the distance between $u_1u_2$ and $v_1v_2$ is at least three:
otherwise by \refclaim{cl-cocyc} and the assumption that $G$ has girth at least five, we would conclude that
$u_1$ or $v_1$ has degree two, which is a contradiction since both have list of size $3a$.
By Theorem~\ref{thm-2flaws}, $\psi$ extends to an $(L':a)$-coloring of $G'$, which also gives an $(L:a)$-coloring of $G$.

If $|L(v_1)|=2a$ (and thus also $|L(u_1)|=2a$), then let $G'=G-R(w)-v_1$, let $\psi$ be an $(L:a)$-coloring of $wv_1$ such that $\psi(w)\cap L(u_1)=\emptyset$,
and let $L'$ be the list assignment obtained from $L$ by choosing $L'(x)$ as a $2a$-element subset of $L(x)\setminus\psi(z)$ for every vertex $x\in V(G)\setminus\{w,v_1\}$ with a neighbor $z$
such that either $z=v_1$, or $z=w$ and $wx\in R(w)$.  By \refclaim{cl-sp} and the assumption that $G$ has girth at least five, $G'$ has no flaws other than $v_2v_3$, and by Theorem~\ref{thm-2flaws}, the restriction of $\psi$ to $w$
extends to an $(L':a)$-coloring of $G'$, which also gives an $(L:a)$-coloring of $G$.

In both cases, we obtain a contradiction.
\end{subproof}

\claim{cl-dist2}{The distance between $C_1$ and $C_2$ is at least $2$.}
\begin{subproof}
Let $C_1=v_1v_2\ldots$ and $C_2=u_1u_2\ldots$, labeled in the clockwise order, and suppose for a contradiction $u_1v_1\in E(G)$.

Suppose first that $|L(v_2)|=3a$.  Let $G'=G-R(u_1v_1)$, let $\psi$ be an $(L:a)$-coloring of
$u_1v_1$ such that if $|L(u_2)|=2a$ then $\psi(u_1)\cap L(u_2)=\emptyset$,
and let $L'$ be the list assignment obtained from $L$ by choosing $L'(x)$ as a $2a$-element subset of $L(x)\setminus\psi(z)$ whenever
$xz\in R(u_1v_1)$ for some $z\in\{u_1,v_1\}$.
By \refclaim{cl-sp} and the assumption that $G$ has girth at least five, $G'$ has no flaws other than $v_2v_3$ and possibly $u_2u_3$ when $|L(u_2)|=3a$
(for example, if $xv_1\in R(u_1v_1)$, then $x$ has no neighbor $r$ in $C_2$ with $|L(r)|=2a$ by \refclaim{cl-sp} applied to the path $u_1v_1xr$).
If either $G'$ has at most one flaw, or the distance between $u_2u_3$ and $v_2v_3$ is at least three, then Theorem~\ref{thm-2flaws} implies
that $\psi$ extends to an $(L':a)$-coloring of $G'$, which also gives an $(L:a)$-coloring of~$G$.

Hence, suppose that both $u_2u_3$ and $v_2v_3$ are flaws (and thus $|L(u_2)|=3a$ and $|L(u_3)|=|L(v_3)|=2a$)
and the distance between them is at most two.
Since $u_2$ and $v_2$ have degree at least three, \refclaim{cl-cocyc} implies that $u_2$ and $v_2$ have a
common neighbor $w\not\in V(C_1\cup C_2)$ and the $5$-cycle $u_1u_2wv_2v_1$ bounds a face.
In that case, let $\psi'$ be an $(L:a)$-coloring of $u_2wv_2$ such that $\psi'(u_2)\cap L(u_3)=\emptyset$
and $\psi'(v_2)\cap L(v_3)=\emptyset$.  Let $G''=G-R(u_2wv_2)$, with list assignment $L''$ obtained
from $L$ by choosing $L''(x)$ as a $2a$-element subset of $L(x)\setminus\psi(z)$ whenever
$xz\in R(u_2wv_2)$ for some $z\in\{u_2,w,v_2\}$.  Since $w$ has degree at least three, \refclaim{cl-cocyc}
together with the assumption that $G$ has girth at least five implies that vertices incident with edges of $R(u_2wv_2)$ form an independent set in $G''$.
Furthermore, these vertices have no other neighbors with list of size $2a$ by \refclaim{cl-sp}
and the fact that $w$ has degree at least three (for example, if $xv_2\in R(u_2wv_2)$, then $x$ has no neighbor $r$ in $C_2$ with $|L(r)|=2a$,
since \refclaim{cl-sp} would imply $u_2wv_2xr$ is a cycle bounding a $5$-face, and thus $w$ would have degree two).  Hence, $G''$ with the list assignment $L''$ has no flaws,
and by Theorem~\ref{thm-2flaws} $\psi'$ extends to an $(L'':a)$-coloring of $G''$.  This gives an $(L:a)$-coloring of $G$, which is a contradiction.

Therefore, for $i\in\{1,2\}$, if $y\in V(C_i)$ is incident with an edge with the other end in $C_{3-i}$, then both neighbors of $y$ in $C_i$ have list of size $2a$,
and consequently $|L(y)|=3a$.

Hence, suppose $|L(v_2)|=|L(u_2)|=2a$.  As observed in the previous paragraph, this implies $v_2$ has no neighbor in $C_2$ and $u_2$ has no neighbor in $C_1$.
By symmetry, we can assume that at least half of the edges of $R(u_1v_1)$ are incident with $v_1$.
Let $G'=G-R(u_1v_1)-v_2$, let $\psi$ be an $(L:a)$-coloring of $u_1v_1v_2$ such that $\psi(u_1)\cap L(u_2)=\emptyset$,
and let $L'$ be the list assignment obtained from $L$ by choosing $L'(x)$ as a $2a$-element subset of $L(x)\setminus\psi(z)$ for every vertex $x\in V(G)\setminus\{u_1, v_1, v_2\}$ with a neighbor $z$
such that either $z=v_2$ or $xz\in R(u_1v_1)$.  If $G'$ with the list assignment $L'$ has no flaw other than $v_3v_4$,
then the restriction of $\psi$ to $u_1v_1$ extends to an $(L':a)$-coloring of $G$ by Theorem~\ref{thm-2flaws}, giving an $(L:a)$-coloring of~$G$.

Hence, assume that $G'$ has a flaw $uv$ other than $v_3v_4$.  Using \refclaim{cl-sp} and the assumption $G$ has girth at least five,
it is straightforward to show that (up to exchange of labels of $u$ and $v$) $u$ is adjacent to $u_1$ and $v$ is adjacent to $v_2$
(for example, if $|L(u)|=2a$ and $vv_2\in E(G)$, then $vu=v_3v_4$ in case $v\in V(C_1)$ by \refclaim{cl-sp} applied to $v_2vu$,
and either $v\in V(C_2)$ or $u=u_2$ in case $u\in V(C_2)$ by \refclaim{cl-sp} applied to $u_1v_1v_2vu$; the former is not possible since $v_2$ has
no neighbor in $C_2$, and in the latter case we have $uu_1,vv_2\in E(G)$).
By \refclaim{cl-cocyc}, the cycle $u_1v_1v_2vu$ bounds a face, and thus $v_1v_2$ is the only edge of $R(u_1v_1)$ incident with $v_1$.
Since at least half of the edges of $R(u_1v_1)$ are incident with $v_1$, it follows that $u_1u_2$ is the only edge of $R(u_1v_1)$ incident with $u_1$,
and thus $u=u_2$.  Since $u_2$ has no neighbor in $C_1$, it follows that $v\neq v_3$.
Note also that $G'$ does not contain any flaw distinct from $uv$ and $v_3v_4$, since we argued that any flaw
distinct from $v_3v_4$ must be incident with the $5$-face of $G'$ whose boundary contains the path $u_1v_1v_2$.
Since $\deg(v)\ge 3$ and $\deg(v_3)\ge 3$, \refclaim{cl-cocyc} together with the assumption that $G$ has girth at least five implies that the distance between $uv$ and $v_3v_4$ in $G'$ is at least three.
Hence, the restriction of $\psi$ to $u_1v_1$ extends to an $(L':a)$-coloring of $G$ by Theorem~\ref{thm-2flaws}, giving an $(L:a)$-coloring of~$G$.
This is a contradiction.
\end{subproof}

\claim{cl-dist3}{The distance between $C_1$ and $C_2$ is at least $3$.}
\begin{subproof}
Let $C_1=v_1v_2\ldots v_s$ and $C_2=u_1u_2\ldots u_t$, labeled in the clockwise order, and suppose for a contradiction $u_1$ and $v_1$ have a common neighbor $w$.
We flip the graph if possible so that $|L(u_2)|=|L(v_2)|=2a$.  If not (i.e., $|L(u_2)|+|L(v_2)|\ge 5a$ and $|L(u_t)|+|L(v_s)|\ge 5a$),
then instead flip the graph so that $w$ is incident with an edge of $R(u_1wv_1)$ (this is possible, since $\deg(w)\ge 3$).

Let $\psi$ be an $(L:a)$-coloring of $u_1wv_1$ such that $\psi(u_1)\cap L(u_2)=\emptyset$ if $|L(u_2)|=2a$ and
and $\psi(v_1)\cap L(v_2)=\emptyset$ if $|L(v_2)|=2a$.  Let $G'=G-R(u_1wv_1)$ and let $L'$ be the list assignment obtained from $L$ by choosing $L'(x)$ as a $2a$-element subset of
$L(x)\setminus\psi(z)$ whenever $xz\in R(u_1wv_1)$ for some $z\in\{u_1,w,v_1\}$.  Suppose first that $|L(u_2)|=|L(v_2)|=2a$.
By \refclaim{cl-sp}, \refclaim{cl-cocyc} and the assumption that $G$ has girth at least five, we conclude that $G'$ contains at most one flaw $uv$, with $u$ adjacent to $u_1$ and $v$ adjacent to $v_1$.
\refclaim{cl-sp} and the assumption that $G$ has girth at least $5$ also implies that this flaw is not connected to $u_1$ and $v_1$.
Hence, Theorem~\ref{thm-2flaws} implies that $\psi$ extends to an $(L':a)$-coloring of $G'$, which also gives an $(L:a)$-coloring of $G$.

Therefore, we can assume that $|L(u_2)|+|L(v_2)|\ge 5a$, and by the choice made in the first paragraph, $w$ is incident with an edge of $R(u_1wv_1)$
and we can assume that $|L(v_s)|=3a$.  By \refclaim{cl-sp}, \refclaim{cl-cocyc} and the assumption
that $w$ is incident with an edge of $R(u_1wv_1)$, $G'$ has no flaws other than $u_2u_3$ and $v_2v_3$. If it has both, then $|L(u_2)|=|L(v_2)|=3a$,
$\deg(u_2),\deg(v_2)\ge 3$, and the distance between the flaws is at least three by \refclaim{cl-cocyc} and the assumption
that $w$ is incident with an edge of $R(u_1wv_1)$.  No flaw in $G'$ is connected to $v_1$, since $|L'(v_s)|=|L(v_s)|=3a$.
Hence, Theorem~\ref{thm-2flaws} implies that $\psi$ extends to an $(L':a)$-coloring of $G'$, which also gives an $(L:a)$-coloring of $G$.

In both cases, we obtain a contradiction.
\end{subproof}

If $C_2$ contained three consecutive vertices with list of size $3a$, we could reduce the list of the middle one to $2a$ without violating the assumptions.
Hence, we can assume that no three consecutive vertices of $C_2$ have list of size $3a$.

Suppose now $C_2$ contains two consecutive vertices with list of size $3a$.
Let $u_1u_2u_3u_4u_5u_6$ be a walk around $C_2$ in clockwise direction (where $u_1=u_6$ if $|C_2|=5$) such that $|L(u_2)|=|L(u_3)|=3a$, and thus $|L(u_1)|=|L(u_4)|=2a$.
If $|L(u_6)|=3a$, then let $X=\{u_4\}$.  If $|L(u_6)|=2a$, then let $X=\{u_4,u_5\}$.  Let $\psi$ be an $(L:a)$-coloring of $G[X]$ such that $\psi(u_5)\cap L(u_6)=\emptyset$ when
$|L(u_6)|=2a$.  Let $L'$ be the list assignment obtained from $L$ by choosing $L'(v)$ as a $2a$-element subset of $L(v)\setminus \psi(x)$ for every vertex $v\in V(G)\setminus X$
with a neighbor $x\in X$.  By the assumption $G$ has girth at least five, \refclaim{cl-sp}, and \refclaim{cl-dist3}, $G-X$ with the list assignment $L'$ has no flaws, and an $(L':a)$-coloring of $G-X$
which exists by the minimality of $G$ together with $\psi$ gives an $(L:a)$-coloring of $G$, which is a contradiction.
We conclude that in $C_2$, the vertices with lists of size $2a$ and $3a$ alternate (and in particular, $|C_2|$ is even).

\claim{cl-dist4}{Let $u_1u_2u_3xy$ be a cycle bounding a $5$-face in $G$, where $u_1u_2u_3$ is a subpath of $C_2$, $|L(u_1)|=|L(u_3)|=3a$ and $|L(u_2)|=2a$.
Then $G$ does not contain a path $u_1ywv_1$ with $|L(v_1)|=2a$.}
\begin{subproof}
By \refclaim{cl-sp}, $v_1\not\in V(C_2)$, and thus $v_1\in V(C_1)$.  Let $C_1=v_1v_2\ldots v_s$ and $C_2=u_1u_2\ldots u_t$, labeled in the clockwise order.
Let $G'=G-R(u_1ywv_1)-u_2$, let $\psi$ be an $(L:a)$-coloring of $u_1ywv_1$ such that $\psi(u_1)\cap L(u_2)=\emptyset$, and let $L'$
be the list assignment for $G'$ obtained from $L$ by setting $L'(v_1)=\psi(v_1)\cup \psi(w)$ and by choosing $L'(v)$ as a $2a$-element subset of
$L(v)\setminus\psi(z)$ whenever $vz\in R(u_1ywv_1)$ for some $z\in\{y,w,v_1\}$.  By \refclaim{cl-cocyc}, \refclaim{cl-sp}, and \refclaim{cl-dist3},
the only possible flaws in $G'$ are $v_2v_3$ and an edge $u'v'$ such that $v_1wyu'v'$ is a $5$-cycle bounding a face of $G$.
Note that $v_1$ is the only neighbor of $w$ with list of size $2a$ in $G'$, and since $G$ has girth at least $5$, we conclude that $w$ is not connected to
either of the flaws.  Suppose that either $G'$ contains at most one flaw or $v'\neq v_2$. Since $\deg(v')\ge 3$ and $\deg(v_2)\ge 3$, \refclaim{cl-cocyc} implies
in the latter case that the distance between the flaws is at least three.  Consequently, the restriction of $\psi$ to $u_1yw$ extends to an $(L':a)$-coloring of $G'$
by Theorem~\ref{thm-2flaws}, giving an $(L:a)$-coloring of $G$.

Hence, $G'$ contains two flaws and $v'=v_2$.  Let $S$ be the set of edges incident with $u_1$ or $y$ distinct from $yu_1$ and not belonging to $R(u_1ywv_1)$, together with the edge $v_1v_2$.
Let $G''=G-S$.  Let $\psi'$ be an $(L:a)$-coloring of $u_1yu'v_2$ such that $\psi'(u_1)\cap L(u_t)=\emptyset$ and $\psi'(v_2)\cap L(v_1)=\emptyset$.
Let $L''$ be the list assignment obtained from $L$ by setting $L''(v_2)=\psi'(v_2)\cup\psi'(u')$ and by choosing $L''(v)$ as a $2a$-element subset of $L(v)\setminus\psi'(z)$ whenever $vz\in S$ for some $z\in\{u_1,y\}$.
By \refclaim{cl-sp} and \refclaim{cl-dist3}, the only flaws in $G''$ are $v_1w$ and $v_2v_3$, and since the girth of $G$ is at least $5$, \refclaim{cl-sp} implies that the distance between these two flaws in $G''$ is at least three.
Furthermore, neither of them is connected to $u_1$ by \refclaim{cl-dist3}.  Hence, the restriction of $\psi'$ to $u_1yu'$ extends to an $(L'',a)$-coloring of $G''$
by Theorem~\ref{thm-2flaws}, giving an $(L:a)$-coloring of $G$.

In both cases, we obtain a contradiction.
\end{subproof}

We now strengthen \refclaim{cl-sp}.
\claim{cl-four}{Let $Q=q_0\ldots q_4$ be a path in $G$ with $q_0,q_4\in V(C_2)$ and $q_1,q_2,q_3\not\in V(C_2)$.  Let $K$ be the cycle consisting
of $Q$ and the path $q_0u_1\ldots u_tq_4$ between $q_0$ and $q_4$ in $C_2$ such that $f_2$ does not lie in $\cin(K)$.  Then
\begin{itemize}
\item $|K|\le 6$ and $K$ bounds a face, or
\item $|K|=8$ and $q_2u_2\in E(\cin(K))$, or
\item $|K|=10$ and $\cin(K)$ consists of $K$, two further vertices $x$ and $y$, and edges $xy$, $xq_1$, $xu_2$, $yu_4$, and $yq_3$.
\end{itemize}}
\begin{subproof}
If $|L(q_0)|=2a$ or $|L(q_4)|=2a$, then $|K|=5$ by \refclaim{cl-sp}; hence, suppose $|L(q_0)|=|L(q_4)|=3a$.  Consequently, $t$ is odd and $|K|$ is even.
If $|K|\le 8$, then the claim follows from \refclaim{cl-cocyc}; hence, suppose $|K|\ge 10$ and $t\ge 5$.  By \refclaim{cl-sp}, we conclude that $K$ is an induced
cycle in $\cin(K)$.

By the minimality of $G$, there exists an $(L:a)$-coloring $\psi$ of $\cex(K)$, and we can extend $\psi$ to $u_1$ as well.
Let $G'=\cin(K)-\{q_0,u_1\}$.  Let $L'(q_1)=\psi(q_2)\cup \psi(q_1)$,  let $L'(z)=L(z)\setminus \psi(r)$ for every vertex $z\in V(G')\setminus\{q_1\}$
with a neighbor $r\in\{q_0,u_1\}$, and let $L'(z)=L(z)$ for any other vertex $z$ of $G'$.  Since $G$ has girth at least five,
the vertices with a neighbor in $\{q_0,u_1\}$ form an independent set, and by \refclaim{cl-sp}, $u_2u_3$ is the only flaw in $G'$.
If this flaw is not connected to $q_2$ or $q_4$, Theorem~\ref{thm-2flaws} implies the restriction of $\psi$ to $q_2q_3q_4$ extends to an $(L':a)$-coloring of $G'$,
which combines with $\psi$ to an $(L:a)$-coloring of $G$, a contradiction.

Hence, we can assume $u_2u_3$ is connected to both $q_2$ and $q_4$.  By \refclaim{cl-sp}, this is only possible if $t=5$, $|K|=10$, and $q_1$ and $u_2$ have a common neighbor $x$.
By symmetry, $q_3$ and $u_4$ have a common neighbor $y$.  By \ref{cl-cocyc} applied to the $8$-cycle $xu_2u_3u_4yq_3q_2q_1$, we have $xy\in E(G)$.  Note that $\cin(K)$ does not
contain any other vertices or edges by \refclaim{cl-cocyc}.
\end{subproof}

A vertex $u\in V(C_2)$ is \emph{rogue} if $|L(u)|=2a$ and either $\deg(u)\ge 3$ or $u$ is not incident with a $5$-face.

\claim{cl-all2prime}{The distance between any two rogue vertices of $C_2$ is at least four, and if $C_2$ contains a rogue vertex, then $|C_2|\ge 8$.}
\begin{subproof}
Let $u_1u_2u_3u_4u_5u_6u_7$ be a walk around $C_2$ in the clockwise direction, with odd-indexed vertices having list of size $2a$,
where $u_3$ is rogue.  Suppose for a contradiction that either $|C_2|=6$ or $u_5$ is rogue.

Let us first consider the case that $|C_2|=6$ and the vertex $u_5$ is not rogue.  Let $X'=\{u_1,\ldots, u_5\}$,
let $\psi'$ be an $(L:a)$-coloring of $u_1u_2u_3u_4$ such that $\psi'(u_4)\cap L(u_5)=\emptyset$,
and let $L''$ be the list assignment for $G''=G-X'$ obtained from $L$ by
choosing $L''(v)$ as a $2a$-element subset of $L(v)\setminus \psi'(x)$ for every vertex $v\in V(G'')$ with
a neighbor $x\in \{u_1,u_2,u_3,u_4\}$.  Note that we do not color the vertex $u_5$;
instead, observe that by the choice of $\psi'(u_4)$ and the fact that $\deg(u_5)=2$, any extension of $\psi'$ to $G-u_5$ extends to an $(L:a)$-coloring of $G$.
Since $u_3$ is rogue, \refclaim{cl-sp} and \refclaim{cl-dist3} imply that
$G''$ has no flaws.  By the minimality of $G$, the graph $G''$ is $(L'':a)$-colorable, which also gives an $(L:a)$-coloring of $G$;
this is a contradiction.

Hence, we can assume that $u_5$ is rogue, and if $|C_2|=6$, then by symmetry also $u_1$ is rogue.
Let $X=\{u_2,\ldots, u_6\}$, and let $\psi$ be an $(L:a)$-coloring of the path $G[X]$ such that $\psi(u_2)\cap L(u_1)=\emptyset$ and $\psi(u_6)\cap L(u_7)=\emptyset$.
Let $G'=G-X$ and let $L'$ be the list assignment for $G'$ obtained from $L$ by choosing $L'(v)$ as a $2a$-element subset of $L(v)\setminus \psi(x)$ for every vertex $v\in V(G')$ with
a neighbor $x\in X$ (note such a neighbor is unique by \refclaim{cl-sp}).  Suppose $uv$ is a flaw in $G'$, where $u\not\in V(C_2)$ and $u$ has a neighbor
$u_i$ in $X$.  By \refclaim{cl-sp} and \refclaim{cl-dist3}, we conclude that $v\not\in V(C_1\cup C_2)$ and $v$ has a neighbor $u_j$ in $X$.
By \refclaim{cl-sp} applied to $u_iuvu_j$, we conclude that $u_i$ and $u_j$ have a common neighbor $u_k$ with list of size $2a$ in $C_2$ and $u_k$ is not rogue.
Since $u_3$ and $u_5$, as well as $u_1$ if $|C_2|=6$, are rogue, it follows that $G'$ does not have any flaws.
By the minimality of $G$, the graph $G'$ is $(L':a)$-colorable. This also gives an $(L:a)$-coloring of $G$, which is a contradiction.
\end{subproof}

\claim{cl-ld}{Every vertex in $C_2$ with list of size $3a$ has degree at least $4$.}
\begin{subproof}
Suppose for a contradiction a vertex $u_4\in V(C_2)$ with list of size $3a$ has degree three, and let $u_1u_2u_3u_4u_5$ be a subpath of $C_2$.  By \refclaim{cl-all2prime}
and symmetry, we can assume $u_3$ is not rogue, and thus there exists a $5$-face bounded by the cycle $u_2u_3u_4uv$.  Subject to these conditions, choose the vertex $u_4$
and the path $u_1u_2u_3u_4u_5$ so that the degree of $u_2$ is as small as possible.

Let $X=\{u_2, u_3, v, u\}$, let $G'=G-X-u_4$,
and let $\psi$ be an $(L:a)$-coloring of the path $G[X]$ such that $\psi(u_2)\cap L(u_1)=\emptyset$ and $|L(u_4)\setminus(\psi(u_3)\cup \psi(u))|\ge 2a$;
such a coloring is obtained by first choosing $\psi(u_2)$ disjoint from $L(u_1)$, and $\psi(u_3)$ disjoint from $\psi(u_2)$,
then choosing $\psi(u)$ with as small intersection with $L(u_4)\setminus\psi(u_3)$ as possible, and finally coloring $v$.
Let $L'$ be the list assignment for $G'$ obtained from $L$ by choosing $L'(v)$ as a $2a$-element subset of $L(v)\setminus \psi(x)$ for every vertex $v\in V(G')$ with
a neighbor $x\in X$.  If $G'$ has no flaws, then by the minimality of $G$, the graph $G'$ is $(L':a)$-colorable, which together with $\psi$ gives an $(L:a)$-coloring of $G-u_4$.
This coloring extends to an $(L:a)$-coloring of $G$ by the choice of $\psi(u)$ and the assumption that $\deg(u_4)=3$, which is a contradiction.

Hence, suppose $G'$ has a flaw $yz$.  \refclaim{cl-sp}, \refclaim{cl-dist3}, and \refclaim{cl-dist4} imply $y,z\not\in V(C_1\cup C_2)$, and thus both $y$ and $z$ have a neighbor in $X$.
Since $G$ has girth at least five, we can assume $uy,u_2z\in E(G)$.  Note that $z\neq v$ since $G$ has girth at least five, and thus $\deg(u_2)\ge 4$.
In particular, by the choice of $u_4$ and the path $u_1u_2u_3u_4u_5$, we conclude that every non-rogue vertex with list of size $2a$ has a neighbor of degree at least four.

Let $K$ be the cycle consisting of the path $u_4uyzu_2$ and the path in $C_2$ between $u_2$ and $u_4$ such that $f_2$ is not contained in $\cin(K)$.
By \refclaim{cl-cocyc}, we have $v\not\in V(\cin(K))$, and thus $u_3\not\in V(K)$.  Since $|C_2|\ge 6$, we have $|K|\ge 8$, and \refclaim{cl-four} implies $u_5$ is a non-rogue vertex
with two neighbors of degree three.  This is a contradiction.
\end{subproof}

Since the last two outcomes of \refclaim{cl-four} are incompatible with \refclaim{cl-ld}, we obtain the following stronger claim.
\claim{cl-four-better}{Let $Q=q_0\ldots q_4$ be a path in $G$ with $q_0,q_4\in V(C_2)$ and $q_1,q_2,q_3\not\in V(C_2)$.  Let $K$ be the cycle consisting
of $Q$ and the path $q_0u_1\ldots u_tq_4$ between $q_0$ and $q_4$ in $C_2$ such that $f_2$ does not lie in $\cin(K)$.  Then $|K|\le 6$ and $K$ bounds a face.}

\claim{cl-all2}{No vertex of $C_2$ is rogue.}
\begin{subproof}
Let $u_1u_2u_3u_4u_5u_6u_7$ be a subpath of $C_2$ with odd-indexed vertices having list of size $2a$,
and suppose for a contradiction that $u_3$ is rogue. By \refclaim{cl-all2prime}, $|C_2|\ge 8$.

Let $X=\{u_2,\ldots, u_6\}$, and
let $\psi$ be an $(L:a)$-coloring of the path $G[X]$ such that $\psi(u_2)\cap L(u_1)=\emptyset$ and $\psi(u_6)\cap L(u_7)=\emptyset$.
Let $G'=G-X$ and let $L'$ be the list assignment for $G'$ obtained from $L$ by choosing $L'(v)$ as a $2a$-element subset of $L(v)\setminus \psi(x)$ for every vertex $v\in V(G')$ with
a neighbor $x\in X$ (note such a neighbor is unique by \refclaim{cl-sp}).  Suppose $uv$ is a flaw in $G'$, where $u\not\in V(C_2)$ and $u$ has a neighbor
$u_i$ in $X$.  By \refclaim{cl-sp} and \refclaim{cl-dist3}, we conclude that $v\not\in V(C_1\cup C_2)$ and $v$ has a neighbor $u_j$ in $X$.
By \refclaim{cl-sp} applied to $u_iuvu_j$, we conclude that $u_i$ and $u_j$ have a common neighbor $u_k$ in $C_2$ of degree two in $G$,
and $u_iuvu_ju_k$ bounds a $5$-face, where $|L(u_k)|=2a$ and $u_k$ is not rogue.   Therefore, we conclude that $uv$ is contained in the $5$-cycle
$u_4u_5u_6uv$ bounding a $5$-face, and thus $G'$ has at most one flaw $uv$.  If $uv$ is at distance at least three from any other vertex with list of size $2a$ in $G'$,
then by the minimality of $G$, the graph $G'$ is $(L':a)$-colorable.  This also gives an $(L:a)$-coloring of $G$, which is a contradiction.

Hence, suppose that $G'$ contains a vertex $w$ at distance two from $uv$ such that $|L'(w)|=2a$.
By \refclaim{cl-sp} and \refclaim{cl-dist4}, we have $w\not\in V(C_1\cup C_2)$, and thus $w$ has a neighbor $u_m\in X$.
Let $y$ denote the common neighbor of $w$ and $\{u,v\}$.  Since $u_4u_5u_6uv$ bounds a $5$-face, \refclaim{cl-four-better} applied to $u_4vywu_m$ or $u_6uywu_m$
implies that $u_4$ or $u_6$ has degree three, contradicting \refclaim{cl-ld}.
\end{subproof}

Now, we can further strengthen \refclaim{cl-four-better}.

\claim{cl-enopath}{There is no path $Q=q_0\ldots q_4$ in $G$ with $q_0,q_4\in V(C_2)$ and $q_1,q_2,q_3\not\in V(C_2)$.}
\begin{subproof}
Suppose such a path exists.  By \refclaim{cl-four-better}, either $q_0q_4\in E(C_2)$, or $q_0$ and $q_4$ have a common neighbor $u$ in $C_2$ and the $6$-cycle $K=q_0q_1q_2q_3q_4u$ bounds
a face distinct from $f_2$.  In the former case, since the sizes of lists in $C_2$ alternate, we can assume that $|L(q_4)|=2a$; but $\deg(q_4)\ge 3$ since $q_4q_3\in E(G)$,
and thus $q_4$ is rogue.  In the latter case, $\deg(u)=2$, and thus $|L(u)|=2a$; but $u$ is not incident with a $5$-face, and thus $u$ is rogue.
In both cases, this contradicts \refclaim{cl-all2}.
\end{subproof}

We are now ready to finish the proof.  Let $C_2=u_1\ldots u_t$, where the odd-indexed vertices have list of size $2a$.
Let $X$ denote the set of vertices of $C_2$ whose index is $2$ or $3$ modulo $4$, and let $Y$ denote the set of vertices whose index is $1$ modulo $4$.
Let $\psi$ be an $(L:a)$-coloring of $G[X]$ such that $\psi(u_i)\cap L(u_{i-1})=\emptyset$ for every $i$ such that $i\equiv 2\pmod 4$.
Let $L'$ be the list assignment for $G-(X\cup Y)$ obtained from $L$ by choosing $L'(v)$ as a $2a$-element subset of $L(v)\setminus \psi(x)$ for every vertex $v\in V(G')$ with
a neighbor $x\in X$.  If $t$ is divisible by $4$, then $G'$ has no flaw by \refclaim{cl-dist3} and \refclaim{cl-sp}.  Otherwise $t\equiv 2\pmod 4$, and by \refclaim{cl-dist3} and
\refclaim{cl-sp} $G'$ only has the flaw $uv$ such that the cycle $u_tu_1u_2uv$ bounds a $5$-face.  This flaw is at distance at least three from any other vertex with list of size $2a$
by \refclaim{cl-dist4} and \refclaim{cl-enopath}.
In either case, the graph $G'$ is $(L':a)$-colorable by the minimality of $G$.  The $(L':a)$-coloring of $G'$ extends to an $(L:a)$-coloring of $G-Y$ by coloring $X$ according to $\psi$.
This coloring further extends to an $(L:a)$-coloring of $G$, since all vertices of $Y$ have degree two and by the choice of $\psi$.
This is a contradiction.
\end{proof}

\section{Strong hyperbolicity}\label{sec-hyper}

We are now ready to prove Theorem~\ref{thm-postrong}.

\begin{proof}[Proof of Theorem~\ref{thm-postrong}]
We follow the proof of Theorem 1.8 in~\cite{postle3crit}.  The only major changes (other than trivial modifications such as replacing all statements
of form $|L(v)|\ge k$ by $|L(v)|\ge ka$, all $L$-colorings by $(L:a)$-colorings, etc.) are as follows.
An analogue of \cite[Lemma~2.9]{postle3crit} is implied directly by Corollary~\ref{cor-distflaws},
so we do not need to prove the analogue of \cite[Theorem 1.17]{postle3crit}; all the other applications of
\cite[Theorem 1.17]{postle3crit} can be replaced by applications of Corollary~\ref{cor-distflaws} as well.
Instead of \cite[Theorem~2.12]{postle3crit}, Theorem~\ref{thm-cyl} is used.  In the statement of \cite[Claim~4.11]{postle3crit},
we replace $A(v)\subseteq A(u)$ by $|A(v)\setminus A(u)|<a$.
We ignore all the claims based on \cite[Claim~4.11]{postle3crit}, which are actually not needed, except for the following one.
In the proof of \cite[Claim~4.17]{postle3crit}, we observe that
by the modified version of \cite[Claim~4.11]{postle3crit}, we can choose $\phi(p_5)$ as an $a$-element subset of $A(p_5)\cap A(p_4)$,
and then since $|A(p_3)|=|A(p_4)|=3a$, we can choose $\phi(p_3)$ as an $a$-element subset of $A(p_3)\setminus (A(p_4)\setminus \phi(p_5))$.
\end{proof}

A more detailed argument can be found in the appendix.

\bibliographystyle{siam}
\bibliography{../data}

\begin{thebibliography}{10}

\bibitem{dvokawalg}
{\sc Z.~Dvo{\v{r}}{\'a}k and K.~Kawarabayashi}, {\em List-coloring embedded
  graphs}, in Proceedings of the Twenty-Fourth Annual {ACM-SIAM} {S}ymposium on
  {D}iscrete {A}lgorithms, {SODA} 2013, {N}ew {O}rleans, {L}ouisiana, {USA},
  {J}anuary 6-8, 2013, SIAM, 2013, pp.~1004--1012.

\bibitem{dmnich}
{\sc Z.~Dvo{\v{r}}{\'a}k and M.~Mnich}, {\em Large independent sets in
  triangle-free planar graphs}, in Algorithms-ESA 2014, Springer, 2014,
  pp.~346--357.

\bibitem{dk}
{\sc Z.~Dvo\v{r}\'ak and K.~Kawarabayashi}, {\em {Choosability of planar graphs
  of girth 5}}, ArXiv, 1109.2976 (2011).

\bibitem{trfree3}
{\sc Z.~Dvo\v{r}\'ak, D.~Kr\'al', and R.~Thomas}, {\em Three-coloring
  triangle-free graphs on surfaces {III}. {G}raphs of girth five}, ArXiv,
  1402.4710 (2014).

\bibitem{trfree4}
\leavevmode\vrule height 2pt depth -1.6pt width 23pt, {\em Three-coloring
  triangle-free graphs on surfaces {IV}. {B}ounding face sizes of $4$-critical
  graphs}, ArXiv, 1404.6356v3 (2015).

\bibitem{trfree7}
{\sc Z.~Dvo\v{r}\'ak, D.~Kr\'al', and R.~Thomas}, {\em Three-coloring
  triangle-free graphs on surfaces {VII}. {A} linear-time algorithm}, ArXiv,
  1601.01197 (2016).

\bibitem{cylgen-part3}
{\sc Z.~Dvo\v{r}\'ak and B.~Lidick\'y}, {\em Fine structure of $4$-critical
  triangle-free graphs {III}. {G}eneral surfaces}, SIAM J. Discrete Math., 32
  (2018), pp.~94--105.

\bibitem{frpltr}
{\sc Z.~Dvo\v{r}\'ak, J.-S. Sereni, and J.~Volec}, {\em Fractional coloring of
  triangle-free planar graphs}, Electronic Journal of Combinatorics, 22 (2015),
  p.~P4.11.

\bibitem{garey1979computers}
{\sc M.~Garey and D.~Johnson}, {\em {Computers and Intractability: A Guide to
  the Theory of {NP}-completeness}}, WH Freeman \& Co. New York, NY, USA, 1979.

\bibitem{grotzsch1959}
{\sc H.~Gr{\"o}tzsch}, {\em Ein {D}reifarbensatz f\"{u}r dreikreisfreie {N}etze
  auf der {K}ugel}, Math.-Natur. Reihe, 8 (1959), pp.~109--120.

\bibitem{Jon84}
{\sc K.~Jones}, {\em Independence in graphs with maximum degree four}, J.
  Combin. Theory Ser. B, 37 (1984), pp.~254--269.

\bibitem{lukethe}
{\sc L.~Postle}, {\em 5-List-Coloring Graphs on Surfaces}, PhD thesis, Georgia
  Institute of Technology, 2012.

\bibitem{postle3crit}
{\sc L.~Postle}, {\em 3-list coloring graphs of girth at least five on
  surfaces}, arXiv, 1710.06898 (2017).

\bibitem{pothom}
{\sc L.~Postle and R.~Thomas}, {\em A {L}inear {U}pper {B}ound for 6-{C}ritical
  {G}raphs on {S}urfaces}.
\newblock Manuscript.

\bibitem{PosThoHyperb}
\leavevmode\vrule height 2pt depth -1.6pt width 23pt, {\em Hyperbolic families
  and coloring graphs on surfaces}, arXiv, 1609.06749 (2016).

\bibitem{thomassen1995-34}
{\sc C.~Thomassen}, {\em 3-list-coloring planar graphs of girth 5}, J. Combin.
  Theory, Ser.~B, 64 (1995), pp.~101--107.

\bibitem{thomassen-surf}
\leavevmode\vrule height 2pt depth -1.6pt width 23pt, {\em The chromatic number
  of a graph of girth 5 on a fixed surface}, J. Combin. Theory, Ser.~B, 87
  (2003), pp.~38--71.

\bibitem{voigt1995}
{\sc M.~Voigt}, {\em A not 3-choosable planar graph without 3-cycles}, Discrete
  Math., 146 (1995), pp.~325--328.

\end{thebibliography}

\section*{Appendix}

It follows from the hyperbolicity theory~\cite{PosThoHyperb} that to establish Theorem~\ref{thm-postrong}, one only needs
to prove the following result, analogous to Theorem~1.8 of~\cite{postle3crit}.

\begin{theorem}\label{thm-hypcyl}
Let $G$ be a plane graph of girth at least five, let $a$ be a positive integer, let $L$ be an assignment of lists of size $3a$
to vertices of $G$, and let $C_1$ and $C_2$ be cycles bounding distinct faces of $G$.
If $G$ is $(a,L,C_1\cup C_2)$-critical, then $|V(G)|\le 89(|C_1|+|C_2|)$.
\end{theorem}

This appendix is devoted to the proof of this theorem. Note that we omit most of the discussion of the
ideas of the argument, as well as straightforward calculations; if necessary, the details can be found in~\cite{postle3crit}.
The fact that (unlike \cite[Theorem~2.12]{postle3crit}) our Theorem~\ref{thm-cyl} has the same general form as Corollary~\ref{cor-distflaws}
means we do not have to distinguish whether the precolored subgraph has one or two components, which streamlines some parts
of the argument; in particular, we can simplify the definition of $\defi(T)$, omitting the term $10(c(S) - c(G))$.
This translates to a slightly better bound in Theorem~\ref{thm-hypcyl} compared to \cite[Theorem~1.8]{postle3crit}.

We need the following standard lemma on critical graphs.
\begin{lemma}\label{lemma-sgcrit}
Let $S$ be a proper subgraph of a graph $G$, let $L$ be a list assignment for $G$, let $a$ be a positive integer,
and let $A$ and $B$ be subgraphs of $G$ such that $G=A\cup B$, $S\subseteq A$, and $B\neq A\cap B$.
If $G$ is $(a,L,S)$-critical, then $B$ is $(a,L,A\cap B)$-critical.
\end{lemma}
\begin{proof}
Consider any subgraph $B'\subsetneq B$ such that $A\cap B\subseteq B'$.  Since $G$ is $(a,L,S)$-critical, there exists an $(L:a)$-coloring $\psi$
of $S$ which extends to an $(L:a)$-coloring $\varphi$ of $A\cup B'$ but not to an $(L:a)$-coloring of $G$.  Then the restriction of $\varphi$ to $A\cap B$
extends to an $(L:a)$-coloring of $B'$ but not to an $(L:a)$-coloring of $B$.  Since this holds for any choice of $B'$, the graph $B$ is $(a,L,A\cap B)$-critical.
\end{proof}

A quadruple $T=(a,G,S,L)$ is a \emph{canvas} if $a$ is a positive integer, $G$ is a plane graph of girth
at least five and $S$ is a subgraph of $G$, and $L$ is a list assignment for $G$ such that $|L(v)|\ge 3a$ for every $v\in V(G)\setminus V(S)$.
The canvas is \emph{critical} if $G$ is $(a,L,S)$-critical (and in particular, this implies that $S\neq G$ and $S$ is $(L:a)$-colorable).  For a subgraph $H\subseteq G$ such that $S\subseteq H$,
the \emph{subcanvas} of $T$ induced by $H$ is the canvas $T|H=(a,H,S,L)$, and the \emph{supercanvas} of $T$ induced by $H$ is the canvas $T/H=(a,G,H,L)$.

\begin{observation}\label{obs-critsc}
Let $T=(a,G,S,L)$ be a canvas.  If some $(L:a)$-coloring of $S$ does not extend to an $(L:a)$-coloring of $G$, then $T$ contains a critical subcanvas.
\end{observation}

Lemma~\ref{lemma-sgcrit} (applied with $A=G-(E(H)\setminus E(S))-(V(H)\setminus V(S))$ and $B=H$, and with $A=H$ and $B=G$) implies the following claims.

\begin{corollary}\label{cor-subsupccrit}
Let $T=(a,G,S,L)$ be a critical canvas and let $H$ be a subgraph of $G$ such that $S\subseteq H$.
\begin{itemize}
\item If $S\neq H$ and all edges of $G$ incident with $V(H)\setminus V(S)$ belong to $H$, then the subcanvas $T|H$ is critical.
\item If $H\neq G$, then the supercanvas $T/H$ is critical.
\end{itemize}
\end{corollary}

Let $T=(a,G,S,L)$ be a canvas and let $P=p_0p_1\ldots p_k$ be a path in $G-V(S)$.
We say $P$ is a \emph{neighboring $k$-path} of $S$ if every vertex of $P$ has a neighbor in $S$,
a \emph{semi-neighboring $3$-path} of $S$ if $k=3$ and $p_0$, $p_1$, and $p_3$ have a neighbor in $S$, and
a \emph{semi-neighboring $5$-path} of $S$ if $k=5$ and $p_0$, $p_1$, $p_4$, and $p_5$ have a neighbor in $S$.
If a vertex $v\in V(G)\setminus V(S)$ has three neighbors $u_1,u_2,u_3\in V(G)\setminus V(S)$, and each of these neighbors
as well as $v$ itself has a neighbor in $S$, then $G[\{v,u_1,u_2,u_3\}]$ is a \emph{neighboring claw} of $S$.

We now prove a combined analogue of Lemmas~2.9 and 2.13 of~\cite{postle3crit}.
\begin{lemma}\label{lemma-neipaths}
If $T = (a,G, S, L)$ is a critical canvas and $S$ has at most two connected components, then $G$ contains one of the
following:
\begin{itemize}
\item[(a)] an edge not in $S$ with both ends in $V(S)$, or
\item[(b)] a vertex not in $V(S)$ with at least two neighbors in $S$, or
\item[(c)] a neighboring 2-path of $S$, or
\item[(d)] a semi-neighboring 3-path of $S$, or
\item[(e)] a semi-neighboring 5-path of $S$.
\end{itemize}
\end{lemma}
\begin{proof}
We can assume $S$ is an induced subgraph of $G$ and every vertex in $V(G)\setminus V(S)$ has at most one neighbor
in $S$, as otherwise (a) or (b) holds.  Let $\psi$ be an arbitrary $(L:a)$-coloring of $S$ that does not
extend to an $(L:a)$-coloring of $G$.  Let $G'=G-V(S)$ and let $L'$ be the list assignment for $G'$
such that $L'(v)$ is a $2a$-element subset of $L(v)\setminus\psi(x)$ for every $v\in V(G')$ with a neighbor $x\in V(S)$,
and $L'(v)$ is a $3a$-element subset of $L(v)$ for any other vertex $v\in V(G')$.  Note that $G'$ is not $(L':a)$-colorable.
Since $S$ has at most two connected components, all vertices of $G'$ with list of size $2a$ are incident with at most two faces of $G'$.
By Theorem~\ref{thm-cyl}, $G'$ contains either a flaw at distance at most two from another vertex with list of size $2a$,
corresponding to the outcome (c) or (d), or two flaws at distance three from each other, corresponding to the outcome (e).
\end{proof}

Let $T = (a,G, S, L)$ be a canvas. Let $P = p_0p_1p_2$ be a neighboring 2-path of $S$ such that for
each $i\in \{0, 1, 2\}$, $p_i$ has a unique neighbor $u_i$ in $S$. Let $H = S + P + \{p_0u_0, p_1u_1, p_2u_2\}$.
We say $T/H$ is obtained from $T$ by \emph{relaxing} $P$ and that $T/H$ is a \emph{relaxation} of $T$. We define $T$ to be a $0$-relaxation of itself.
For a positive integer $k$, we say a supercanvas $T'$ of $T$ is a \emph{$k$-relaxation} of $T$ if $T'$ is a relaxation of a $(k-1)$-relaxation of $T$.

Let us now prove a combined analogue of Lemmas~2.11 and 2.14 of~\cite{postle3crit}.
\begin{lemma}\label{lemma-neiparel}
If $T = (a,G, S, L)$ is a critical canvas and $S$ has at most two connected components, then there exists one of the
following:
\begin{itemize}
\item[(1)] an edge of $G$ not in $S$ with both ends in $V(S)$, or
\item[(2)] a vertex of $G$ not in $V(S)$ with at least two neighbors in $S$, or
\item[(3)] a neighboring claw of $S$, or
\item[(4)] a $(\le\!2)$-relaxation $(a,G,S',L)$ of $T$ and a semi-neighboring 3-path of $S'$, or
\item[(5)] a $(\le\!2)$-relaxation $(a,G,S',L)$ of $T$ and a semi-neighboring 5-path of $S'$.
\end{itemize}
\end{lemma}
\begin{proof}
We can assume $S$ is an induced subgraph of $G$, every vertex in $V(G)\setminus V(S)$ has at most one neighbor
in $S$, and $G$ does not contain a semi-neighboring 3-path of $S$, and thus also no neighboring $3$-path of $S$.
Let $N\subseteq V(G)\setminus V(S)$ consist of all vertices with a neighbor in $S$; since $G$ does not contain
a neighboring $3$-path and has girth at least five, each component of $G[N]$ is a star.  If $G[N]$ contains a vertex of degree at least three,
then $G$ contains a neighboring claw of $S$ and (3) holds.  Hence, we can assume $G[N]$ is a union of paths of length at most two.
Let $R\subseteq N$ consist of vertices of the length-$2$ paths of $G[N]$.  By Lemma~\ref{lemma-neipaths}, we can assume $R$ is non-empty.
For each vertex $u\in R$, let $P(u)$ be the path of $G[R]$ that contains it.  Let $H=G[V(S)\cup R]$.

Consider any path $Q=xy$ in $G-V(H)$ such that $x$ has a neighbor $u\in R$ and $y$ has a neighbor $v\in V(H)$.
If $v\in V(S)$, then a semi-neighboring 3-path of $S$ is contained in $G[V(P(u))\cup\{x,y\}]$, and (4) holds.
If $v\in R$ and $P(v)\neq P(u)$, then a semi-neighboring 5-path of $S$ is contained in $G[V(P(u))\cup\{x,y\}\cup V(P(v))]$, and (5) holds.
Hence, suppose this is not the case, and thus for any such path $Q$, we have $v\in R$ and $P(u)=P(v)$.  Let us define $P(Q)=P(u)$.

Note that $H\neq G$, since the endvertices of the paths in $G[R]$ have degree at least three
in $G$ (since $G$ is $(a,L,S)$-critical and these vertices have lists of size at least $3a$) but degree two in $H$.
Corollary~\ref{cor-subsupccrit} implies $T/H$ is a critical canvas.  Note that $H$ has at most two connected components,
and thus we can apply Lemma~\ref{lemma-neipaths} to $T/H$.  Since $H$ is an induced subgraph of $G$, (a) does not occur.

Suppose (b) occurs, that is, a vertex $v\in V(G)\setminus V(H)$ has neighbors $u_1,u_2\in V(H)$.
Since every vertex in $V(G)\setminus V(S)$ has at most one neighbor in $S$ and $R$ is the union of some of the connected components of $G[N]$,
we conclude that $u_1,u_2\in R$, and since $G$ has girth at least five, we have $P(u_1)\neq P(u_2)$.  But then
$G$ contains a semi-neighboring 3-path of $S$ and (4) holds.

Suppose (c) occurs, and let $P=p_0p_1p_2$ be a neighboring 2-path of $H$.  For $i\in \{0, 1, 2\}$, let $u_i$ be a neighbor of $p_i$ in $H$.
Since $V(P)\cap R=\emptyset$, at least one of these neighbors belongs to $R$; by symmetry we can assume $\{u_0,u_1\}\cap R\neq\emptyset$.
As we argued earlier, this implies $u_0$ and $u_1$ belong to the same path $P(p_0p_1)$ in $G[R]$.  But then $u_1\in R$, and thus
also $u_1$ and $u_2$ belong to the same path $P(p_1p_2)$ in $G[R]$, and clearly $P(p_0p_1)=P(p_1p_2)$.  But since $G$ has girth at least
five, the vertices $u_0$, $u_1$, and $u_2$ cannot all belong to the same path of length three, which is a contradiction.
Hence, (c) does not occur.

Suppose (d) occurs, and let $P=p_0p_1p_2p_3$ be a semi-neighboring 3-path of $H$; for $i\in \{0, 1, 3\}$, let $u_i$ be a neighbor of $p_i$ in $H$.
Note that either $u_0,u_1\in V(S)$, or $u_0$ and $u_1$ are the endpoints of the path $P(p_0p_1)$ in $G[R]$.
Hence, $P$ is a semi-neighboring 3-path of $S'$ in a $(\le\!2)$-relaxation $(a,G,S',L)$ of $T$ obtained by relaxing at most two paths
$P(p_0p_1)$ and $P(u_3)$.  Hence, (4) holds.

Similarly, if (e) occurs, then a semi-neighboring 5-path $P=p_0\ldots p_5$ of $H$ is a semi-neighboring $5$-path of $S'$
in a $(\le\!2)$-relaxation $(a,G,S',L)$ of $T$ obtained by relaxing at most two paths $P(p_0p_1)$ and $P(p_4p_5)$.
\end{proof}

Let $T = (a,G, S, L)$ be a canvas.  We let $v(T)=|V(G) \setminus V(S)|$ and $e(T) = |E(G) \setminus E(S)|$.  
We define the \emph{deficiency} of the canvas $T$ as $\defi(T) = 3e(T) - 5v(T)$.
Let $q(T)=\sum_{v\in V(S)} \deg_{G-E(S)}(v)$
be the number of edges not in $E(S)$ with an end in $V(S)$, where the edges with both ends in $V(S)$ are counted twice.
Let $\alpha=3/8$ and $\varepsilon=1/88$.  Let $s(T)=\varepsilon v(T) + \alpha q(T)$, and let $d(T) = \defi(T)-s(T)$.
For a subgraph $H$ of $G$ containing $S$, let $q_T(H,S)=\sum_{v\in V(H)\setminus V(S)} \deg_{G-E(H)}(v)$ and let $d_T(T|H) = d(T|H) + \alpha q_T(H,S)$.
The following claims are established by straightforward calculations.

\begin{lemma}[{Postle~\cite[Lemma~3.2, Proposition 3.4, and Proposition 3.6]{postle3crit}}]\label{lemma-addit}
If $T=(a,G, S, L)$ is a canvas and $H$ is a subgraph of $G$ containing $S$, then
\begin{itemize}
\item $\defi(T) = \defi(T|H) + \defi(T/H)$,
\item $v(T) = v(T|H) + v(T/H)$,
\item $e(T) = e(T|H) + e(T/H)$,
\item $q(T) = q(T|H) + q(T/H) - q_T(H,S)\le q(T|H) + q(T/H)$,
\item $s(T) \le s(T|H) + s(T/H)$, and
\item $d(T) = d_T(T|H) + d(T/H) \ge d(T|H) + d(T/H)$.
\end{itemize}
\end{lemma}

\begin{lemma}[{Postle~\cite[Proposition 3.7]{postle3crit}}]\label{lemma-small}
Let $T = (a, G, S, L)$ be a canvas.
\begin{itemize}
\item[(i)] If $G = S$, then $d(T) = 0$.
\item[(ii)] If $v(T) = 0$, then $d(T)=(3 - 2\alpha)e(T)$.
\item[(iii)] If $v(T) = 1$, then $d(T)=(3 - 2\alpha)e(T) - 5 + \alpha\deg(v) - \varepsilon$,
where $v$ is the unique vertex in $V(G)\setminus V(S)$.
\end{itemize}
\end{lemma}

Let $T = (a, G, S, L)$ be a canvas. We say $T$ is a \emph{chord} if $G$ consists of exactly $S$ and one
edge not in $S$ with both ends in $S$. We say $T$ is a \emph{tripod} if $G$ consists of exactly $S$ and one vertex not in $S$
with three neighbors in $S$. We say $T$ is \emph{singular} if $T$ is a chord or a tripod and \emph{non-singular} otherwise. We say $T$ is
\emph{normal} if no subcanvas of $T$ is singular.  By Lemma~\ref{lemma-small}, if $T$ is a chord, then $d(T)=3 - 2\alpha$,
and if $T$ is a tripod, then $d(T)=4 - 3\alpha - \varepsilon$.  Let $\gamma(T)$ be the triple $\Bigl(v(T),e(T), \sum_{v\in V(G)\setminus V(S)} |L(v)|\Bigr)$.
For canvases $T_1$ and $T_2$, we write $T_1\prec T_2$ if $\gamma(T_1)$ is lexicographically smaller than $\gamma(T_2)$, and $T_1\preceq T_2$ if
$T_1\prec T_2$ or $\gamma(T_1)=\gamma(T_2)$.

For a graph $H$, let $c(H)$ denote the number of components of $H$.
We now prove the key result, corresponding to Theorem~3.9 of~\cite{postle3crit}.

\begin{theorem}\label{thm-canvas}
If $T_0 = (a, G_0, S_0, L_0)$ is a non-singular critical canvas and $c(S_0)\le 2$, then $d(T_0)\ge 3$.
\end{theorem}
\begin{proof}
Suppose for a contradiction that $T_0$ is a minimal counterexample, that is,
$d(T_0)<3$, but for every critical canvas $T=(a,G,S,L)\prec T_0$ with $c(S)\le 2$,
if $T$ is non-singular, then $d(T)\ge 3$.
Note that $3 > 4 - 3\alpha - \varepsilon > 3 - 2\alpha$, and thus $d(T)\ge 3 - 2\alpha$
even if $T$ is singular.

We say a canvas $T=(a,G,S,L)$ is \emph{close} to $T_0$ if $T\preceq T_0$, $T$ is critical, $c(S)\le 2$, and $d(T) \le  d(T_0) + 6\varepsilon$.
Note this implies $d(T)<3+6\varepsilon$.
The following general bound on $d_T(T|H)$ will be quite useful.

\claim{cl-dtth}{Let a canvas $T=(a,G,S,L)$ be close to $T_0$, and let $H$ be a proper subgraph of $G$ such that $S\subsetneq H$ and $c(H)\le 2$.
Then $d_T(T|H)<2\alpha+6\varepsilon$, if $T/H$ is not a chord, then $d_T(T|H)<3\alpha+7\varepsilon-1$, and if $T/H$ is non-singular, then $d_T(T|H)<6\varepsilon$.}
\begin{subproof}
Note that $T/H$ is critical by Corollary~\ref{cor-subsupccrit}, and since $T/H\prec T_0$, we have $d(T/H)\ge 3 - 2\alpha$.
By Lemma~\ref{lemma-addit}, $d(T)=d_T(T|H)+d(T/H)$, and since $d(T)<3+6\varepsilon$, it follows that $d_T(T|H)<2\alpha+6\varepsilon$.
If $T/H$ is not a chord, then $d(T/H)\ge 4-3\alpha-\varepsilon$, and thus the same inequality gives $d_T(T|H)<3\alpha+7\varepsilon-1$,
and if $T/H$ is non-singular, then $d(T/H)\ge 3$, giving $d_T(T|H)<6\varepsilon$.
\end{subproof}

This has a number of simple corollaries.
\claim{cl-cors}{Let a canvas $T=(a,G,S,L)$ be close to $T_0$.  Then
\begin{itemize}
\item[(i)] either $T$ is a chord or $S$ is an induced subgraph of $G$,
\item[(ii)] either $T$ is a tripod or each vertex $v\in V(G)\setminus V(S)$ has at most one neighbor in $S$,
\item[(iii)] $G$ does not contain a neighboring $3$-path $P$ of $S$, and
\item[(iv)] $G$ does not contain a neighboring claw $C$ of $S$.
\end{itemize}}
\begin{subproof}
Let us discuss the configurations one by one.
\begin{itemize}
\item[(i)] Suppose $G$ contains an edge $e\not\in E(S)$ with both ends in $S$, and let $H=S+e$.
Note that $d_T(T|H)=3-2\alpha>2\alpha+6\varepsilon$.
We conclude that $G=H$ by \refclaim{cl-dtth}; hence, $T$ is a chord.
\item[(ii)] Suppose a vertex $v\in V(G)\setminus V(S)$ has distinct neighbors $u_1,u_2\in v(S)$,
and let $H=S+u_1vu_2$.  Since $v$ has degree at least three in $G$, note that $H\neq G$, $q_T(H,S)\ge 1$, and
$d_T(T|H)\ge 1-\alpha-\varepsilon>3\alpha+7\varepsilon-1$,
and thus \refclaim{cl-dtth} implies $T/H$ is a chord.  Consequently, the neighbor of $v$ distinct from $u_1$ and $u_2$
also belongs to $S$, and $T$ is a tripod.
\item[(iii)] Consider a neighboring $3$-path $P$ of $S$; by (ii), each vertex $p\in V(P)$ has exactly one neighbor $v_p\in V(S)$.
Let $H=S+P+\{pv_p:p\in V(P)\}$.  Note that the endvertices of $P$ have degree at least three in $G$ and two in $H$, and they are
non-adjacent since the girth of $G$ is at least five.  Hence, $G\neq H$, $T/H$ is not a chord, $q_T(H,S)\ge 2$, and
$d_T(T|H)\ge 1-2\alpha-4\varepsilon=3\alpha+7\varepsilon-1$, which contradicts \refclaim{cl-dtth}.
\item[(iv)] Consider a neighboring claw $C$ of $S$; by (ii), each vertex $p\in V(C)$ has exactly one neighbor $v_p\in V(S)$.
Let $H=S+C+\{pv_p:p\in V(C)\}$.  Note that the rays of $C$ have degree at least three in $G$ and two in $H$,
and they do not have a common neighbor in $V(G)\setminus V(H)$ since the girth of $G$ is at least five.
Hence, $T/H$ is non-singular, $q_T(H,S)\ge 3$, and $d_T(T|H)\ge 1-\alpha-4\varepsilon>6\varepsilon$, which contradicts \refclaim{cl-dtth}.
\end{itemize}
\end{subproof}

Next, we give another general bound.

\claim{cl-normal}{Suppose $T_1 = (a,G_1, S_1, L_1)\preceq T_0$ is a normal critical canvas such that $c(S_1) \le 2$.
Let $H_1$ be a proper subgraph of $G_1$ containing $S_1$, and let $L'$ be a list assignment for $H_1$
such that $|L'(v)|\ge 3a$ for every $v\in V(H_1)\setminus V(S_1)$.  If $H_1$ is $(a,L',S_1)$-critical, then
$d(T_1) \ge 6 - \alpha$. Furthermore, if $|E(G_1) \setminus E(H_1)| \ge 2$, then $d(T_1) \ge 6$.}
\begin{subproof}
Since $H_1$ is $(a,L',S_1)$-critical, by Corollary~\ref{cor-distflaws} every component of $H_1$ intersects $S_1$,
and thus $c(H_1)\le c(S_1)\le 2$.
Consider the canvas $T'=(a,H_1,S_1,L')$.  Since $T_1\preceq T_0$ and $H_1$ is a proper subgraph of $G_1$,
we conclude that $T'\prec T_0$.  Since $T_1$ is normal, $T'$ is non-singular, and thus $d(T')\ge 3$.
Hence, $d(T_1|H_1) = d(T'_1)\ge 3$.

Since $H_1$ is a proper subgraph of $G_1$, $T_1/H_1$ is critical by Corollary~\ref{cor-subsupccrit}.
If $T_1/H_1$ is non-singular, then since $T_1/H_1\prec T_0$, we have $d(T_1/H_1)\ge 3$,
and thus $d(T_1)\ge d(T_1|H_1)+d(T_1/H_1)\ge 6$ by Lemma~\ref{lemma-addit}.

If $T_1/H_1$ is singular, then since $T_1$ is normal, $E(G_1)\setminus E(H_1)$ must contain an edge with an end in $V(H_1)\setminus V(S_1)$,
and thus $q_{T_1}(H_1,S_1)\ge 1$ and $d_{T_1}(T_1|H_1)\ge 3+\alpha$.  Hence, if $T_1/H_1$ is a tripod, implying $d(T_1/H_1)\ge 4 - 3\alpha - \varepsilon$,
we have $d(T_1)\ge d(T_1|H_1)+d_{T_1}(T_1/H_1)\ge 7-2\alpha-\varepsilon>6$ by Lemma~\ref{lemma-addit}.  And if $T_1/H_1$ is a chord, implying
$d(T_1/H_1)\ge 3-2\alpha$ and $|E(G_1) \setminus E(H_1)|=1$, we have $d(T_1)\ge d(T_1|H_1)+d_{T_1}(T_1/H_1)\ge 6-\alpha$.
\end{subproof}

We now apply this claim to canvases close to $T_0$.
\claim{cl-nosg}{Let a canvas $T=(a,G,S,L)$ be close to $T_0$, let $H$ be a proper subgraph of $G$ containing $S$, and let $L'$ be
a list assignment for $H$ such that $|L'(v)|\ge 3a$ for every $v\in V(H)\setminus V(S)$.  Then $H$ is not $(a,L',S)$-critical.}
\begin{subproof}
Suppose for a contradiction that $H$ is $(a,L',S)$-critical, and in particular $H\neq S$ and vertices of $V(H)\setminus V(S)$
have degree at least three in $H$.  Since $H\neq G$, we conclude that
$T$ is non-singular, and by \refclaim{cl-cors}(i) and (ii), $T$ is normal.  By \refclaim{cl-normal}, we conclude $d(T)\ge 6-\alpha$.
However, since $T$ is close to $T_0$, we have $d(T)<3+6\varepsilon$, which is a contradiction.
\end{subproof}

Let us give another application of \refclaim{cl-normal}, which we use to reduce semi-neighboring paths.
\claim{cl-reduce}{Let a canvas $T=(a,G,S,L)$ be close to $T_0$.  Let $H$ be a proper subgraph of $G$ containing $S$ such that $c(H)\le 2$ and the canvas $T/H$ is normal. If there exists a proper
subgraph $G'$ of $G$ such that $H \subseteq G'$ and there exists an $(L:a)$-coloring of $H$ that does not extend to an $(L:a)$-coloring of $G'$,
then $d_T(T|H) < -3 + \alpha + 6\varepsilon$. Furthermore, if $|E(G) \setminus E(G')|\ge 2$, then $d_T(T|H) < -3 + 6\varepsilon$.}
\begin{subproof}
Since there exists an $(L:a)$-coloring of $H$ that does not extend to an $(L:a)$-coloring of $G'$, some subgraph $G''\subseteq G'$
containing $H$ is $(a,L,H)$-critical.  By \refclaim{cl-normal} applied with $T_1=T/H$, $H_1=G''$, and $L'=L$,
we have $d(T/H)\ge 6-\alpha$, and if $|E(G) \setminus E(G')|\ge 2$, then $d(T/H)\ge 6$.
Since $T$ is close to $T_0$, we have $d(T)<3+6\varepsilon$.  By Lemma~\ref{lemma-addit},
$d_T(T|H)=d(T)-d(T/H)$, implying the claimed inequalities.
\end{subproof}

We say that a canvas $T=(a,G,S,L)$ is \emph{nice} if there exists an $(L:a)$-coloring $\psi_L$ of $S$
such that $L(v)=\psi_L(v)$ for every $v\in V(S)$, $|L(z)|=3a$ for every $z\in V(G)\setminus V(S)$, and if additionally $z$ has exactly one neighbor $v\in V(S)$,
then $\psi_L(v)\subset L(z)$.  Note that if $T$ is critical and nice, then $G$ is not $(L:a)$-colorable.
\claim{cl-hasnice}{If a canvas $T=(a,G,S,L)$ is close to $T_0$, then there exists a list assignment $L'$ such that the canvas $T'=(a,G,S,L')$
is nice and close to $T_0$.}
\begin{subproof}
Since $T$ is critical, there exists an $(L:a)$-coloring $\psi$ of $S$ that does not extend to an $(L:a)$-coloring of $G$.
Let $L'(v)=\psi(v)$ for $v\in V(S)$, let $L'(z)$ be a subset of $L(z)\cup \psi(v)$ of size $3a$ and containing $\psi(v)$ for every $z\in V(G)\setminus V(S)$
with a unique neighbor $v\in V(S)$, and let $L'(z)$ be a subset of $L(z)$ of size $3a$ for every other vertex $z\in V(G)\setminus V(S)$.
Clearly $T'$ is nice.  Since $|L'(v)|\le |L(v)|$ for every $v\in V(G)$, we have $T'\preceq T_0$.
Furthermore, $\psi$ extends to an $L'$-coloring of every proper subgraph of $G$ containing $S$, as otherwise there
would exist a proper $(a,L',S)$-critical subgraph of $G$ containing $S$, contradicting \refclaim{cl-nosg}.
Hence, $T'$ is critical, and thus $T'$ is close to $T_0$.
\end{subproof}

For a nice canvas $T=(a,G,S,L)$ and a vertex $v\in V(G)\setminus V(S)$, let $A_T(v)=L(v)\setminus \bigcup_{u\in V(S), uv\in E(G)} L(u)$
be the set of colors in $L(v)$ that are not excluded by the coloring $\psi_L$.  Note that if $v$ has exactly one neighbor in $S$, then $|A_T(v)|=2a$.

\claim{cl-adj}{Let $T=(a,G,S,L)$ be a nice canvas close to $T_0$.  Suppose $v\in V(G)\setminus V(S)$ has a neighbor $z\in V(S)$, and let $u\in V(G)\setminus V(S)$
be another neighbor of $v$.  Then $|A_T(v)\setminus A_T(u)|<a$.}
\begin{subproof}
Suppose for a contradiction that there exists an $a$-element set $c\subseteq A_T(v)\setminus A_T(u)$.  Let $H=S+zv$ and let $\psi$ be an $(L:a)$-coloring of $H$
obtained from $\psi_L$ by setting $\psi(z)=c$.  Since $v$ has degree at least three in $G$ and degree one in $H$, we have and $d_T(T|H)\ge -2-\varepsilon+\alpha>-3 + \alpha + 6\varepsilon$.
Because of the vertex $u$, $T$ is neither a chord nor a tripod, and thus by \refclaim{cl-cors}(i) and (ii), $S$ is an induced subgraph of $G$ and no vertex $x\in V(G)\setminus V(S)$ has
more than one neighbor in $S$; we conclude that the canvas $T/H$ is normal.  Letting $G'=G-uv$, \refclaim{cl-reduce} implies $\psi_L$ extends to an $(L:a)$-coloring $\varphi$ of $G'$.
We have $\varphi(u)\subseteq A_T(u)$, and thus $c\cap\varphi(u)=\emptyset$.  Consequently, $\varphi$ is an $(L:a)$-coloring of $G$, which is a contradiction.
\end{subproof}

Let us now bound degrees of vertices in neighboring $1$-paths and $2$-paths.
\claim{cl-1path}{Let $T=(a,G,S,L)$ be a canvas close to $T_0$.  If $P = p_0p_1$ is a neighboring $1$-path of $S$ in $G$ and $P$ is not a subpath of a neighboring $2$-path,
then either $\deg(p_0) \ge 4$ or $\deg(p_1) \ge 4$.}
\begin{subproof}
By \refclaim{cl-hasnice}, we can assume that $T$ is nice.
Suppose for a contradiction that $\deg(p_0) = \deg(p_1) = 3$.  For $i\in \{0,1\}$, let $z_i$ and $u_i$ be the neighbors of $p_i$ distinct from $p_{1-i}$
such that $z_i\in V(S)$.  By \refclaim{cl-cors}(ii), we have $u_0,u_1\not\in V(S)$.  Since $P$ is not a subpath of a neighboring $2$-path,
we have $A_T(u_i)=L(u_i)$ for $i\in \{0,1\}$.  By symmetry, assume $\deg(u_0)\ge \deg(u_1)$.  Let $H=S+u_0p_0p_1z_1+p_0z_0$.  Note that $\deg_{G-E(H)}(u_0)+\deg_{G-E(H)}(p_1)=\deg_G(u_0)$,
and thus $d_T(T|H)\ge-3+(\deg_G(u_0)-2)\alpha-3\varepsilon>3+6\varepsilon$, and $d_T(T|H)>-3 + \alpha + 6\varepsilon$ unless $\deg_G(u_0)=3$.
By \refclaim{cl-cors}(i) and (ii) and since $G$ has girth at least five, the canvas $T/H$ is normal.

Let $G'=G-p_1u_1$ if $\deg(u_0)>3$ and $G'=G-u_1$ if $\deg(u_0)=3$.  Recall that $|A_T(p_0)|=2a$, since $T$ is nice and $p_0$ has exactly one neighbor in $S$.
Choose $a$-element sets $\psi(u_0)\subseteq L(u_0)\setminus A_T(p_0)$, $\psi(p_0)\subseteq A_T(p_0)\setminus\psi(u_0)$, and $\psi(p_1)\subseteq A_T(p_1)\setminus\psi(p_0)$
arbitrarily.  Let $\psi(v)=\psi_L(v)$ for $v\in V(S)$.  Note that if $\deg_G(u_0)=3$, then $|E(G)\setminus E(G')|>2$; hence, \refclaim{cl-reduce} implies
that $\psi$ extends to an $(L:a)$-coloring $\varphi$ of $G'$.   If $\deg(u_0)=3$, then since $\deg(u_0)\ge \deg(u_1)$, we have $\deg(u_1)=3$, and thus $\varphi$ extends to an $(L:a)$-coloring of $G-p_1u_1$.
Change in $\varphi$ the color set of $p_1$ to an $a$-element subset $c_1$ of $A_T(p_1)\setminus\varphi(u_1)$, and change the color set of $p_0$ to an $a$-element subset
of $A_T(p_0)\setminus c_1$.  By the choice of $\psi(u_0)$, this gives an $(L:a)$-coloring of $G$, a contradiction.
\end{subproof}

\claim{cl-2path}{Let $T=(a,G,S,L)$ be a canvas close to $T_0$.  If $P = p_0p_1p_2$ is a neighboring $2$-path of $S$ in $G$, then $\deg(p_0)+\deg(p_1)+\deg(p_2)\ge 10$.}
\begin{subproof}
By \refclaim{cl-hasnice}, we can assume that $T$ is nice.
Suppose for a contradiction that $\deg(p_0) = \deg(p_1) = \deg(p_2)=3$.  For $i\in \{0,1,2\}$, let $z_i$ be the neighbors of $p_i$ in $S$
(unique by \refclaim{cl-cors}(ii)), and for $i\in \{0,2\}$, let $u_i$ be the neighbor of $p_1$ distinct from $z_i$ and $p_1$.
By \refclaim{cl-cors}(iii), we have $A_T(u_i)=L(u_i)$ for $i\in \{0,2\}$.

Let $H=S+u_0p_0p_1p_2z_2+p_0z_0+p_1z_1$.  Note that $\deg_{G-E(H)}(u_0)+\deg_{G-E(H)}(p_1)=\deg_G(u_0)$,
and thus $d_T(T|H)\ge -2+(\deg_G(u_0)-3)\alpha-4\varepsilon>-3 + \alpha + 6\varepsilon$.
By \refclaim{cl-cors}(i), (ii), and (iii), and since $G$ has girth at least five, the canvas $T/H$ is normal.

Let $G'=G-p_2u_2$.  Choose $a$-element sets $\psi(u_0)\subseteq L(u_0)\setminus A_T(p_0)$, $\psi(p_0)\subseteq A_T(p_0)\setminus\psi(u_0)$, $\psi(p_1)\subseteq A_T(p_1)\setminus\psi(p_0)$,
and $\psi(p_2)\subseteq A_T(p_2)\setminus\psi(p_1)$ arbitrarily.  Let $\psi(v)=\psi_L(v)$ for $v\in V(S)$.  \refclaim{cl-reduce} implies
that $\psi$ extends to an $(L:a)$-coloring $\varphi$ of $G'$.  Change in $\varphi$ the color set of $p_2$ to an $a$-element subset $c_2$ of $A_T(p_2)\setminus\varphi(u_2)$,
the color set of $p_1$ to an $a$-element subset $c_1$ of $A_T(p_1)\setminus c_2$, and the color set of $p_0$ to an $a$-element subset
of $A_T(p_0)\setminus c_1$.  By the choice of $\psi(u_0)$, this gives an $(L:a)$-coloring of $G$, a contradiction.
\end{subproof}

Next, let us argue about relaxations.
\claim{cl-relax}{Let $T=(a,G,S,L)$ be a canvas close to $T_0$.  If $T'$ is a relaxation of $T$, then $T'$ is critical and $d(T')\le d(T)+3\varepsilon$.}
\begin{subproof}
Suppose $T'$ is obtained from $T$ by relaxing a $2$-neighboring path $P=p_0p_1p_2$.  Let $H$ be the subgraph of $G$ consisting of $S$, $P$, and the edges
between them, so that $T'=T/H$.  Note that the endpoints of $P$ have degree at least three in $G$ but two in $H$, and thus $H\neq G$.
Hence, $T'$ is critical by Corollary~\ref{cor-subsupccrit}.  By \refclaim{cl-2path}, we have $q_T(H,S)\ge 3$, and thus $d_T(T|H)\ge -3\varepsilon$.
By Lemma~\ref{lemma-addit}, $d(T')=d(T/H)=d(T)-d_T(T|H)\le d(T)+3\varepsilon$.
\end{subproof}

Iterating \refclaim{cl-relax}, we obtain the following corollary.
\claim{cl-close}{For $k\le 2$, every $k$-relaxation of $T_0$ is close to $T_0$.}

We now deal with semi-neighboring paths.
\claim{cl-semi3path}{If $T=(a,G,S,L)$ is a canvas close to $T_0$, then $G$ does not contain a semi-neighboring $3$-path of $S$.}
\begin{subproof}
By \refclaim{cl-hasnice}, we can assume that $T$ is nice.
Suppose for a contradiction $P=p_0p_1p_2p_3$ is a semi-neighboring $3$-path of $S$ in $G$.  For $i\in\{0,1,3\}$, the neighbor of $p_i$ in $S$
is unique by \refclaim{cl-cors}(ii).  Let $R$ be the set of vertices in $V(G)\setminus (V(S)\cup\{p_0,p_1\})$ with a neighbor both in $\{p_0,p_1\}$ and in $S$.
Note that $p_2\not\in R$ and $|R|\le 1$ by \refclaim{cl-cors}(iii) and (iv).  Since $G$ has girth at least five,
there is at most one edge between $\{p_0,p_1,p_2\}$ and $R$.  Let $H=G[V(S)\cup R\cup \{p_0,p_1,p_2\}]$.  We have
$d_T(T|H)\ge -3+2\alpha-3\varepsilon$ if $R=\emptyset$ using \refclaim{cl-1path}, and
$d_T(T|H)\ge -2+\alpha-4\varepsilon$ if $R\neq\emptyset$ using \refclaim{cl-2path}.
In either case, $d_T(T|H)>-3 + \alpha + 6\varepsilon$.
Note also that the canvas $T/H$ is normal: $H$ is an induced subgraph, no vertex can have three neighbors in $R\cup\{p_0,p_1,p_2\}$ since $G$ has girth at least
five, no vertex $v\in V(G)\setminus V(S)$ can have more than one neighbor in $S$ by \refclaim{cl-cors}(ii), and $v\in V(G)\setminus V(H)$ cannot have two neighbors
in $R\cup\{p_0,p_1,p_2\}$ and one neighbor in $S$: since $v\not\in R$, $v$ would be adjacent the vertex of $R$ (and to $p_2$), contradicting \refclaim{cl-cors}(iii).

Since $|A_T(p_3)|=2a$ and $|A_T(p_2)|=3a$, there exists an $(L:a)$-coloring $\psi$ of $H$ extending $\psi_L$ such that
$\psi(p_2)\cap A_T(p_3)=\emptyset$.  By \refclaim{cl-reduce}, $\psi$ extends to an $(L:a)$-coloring of $G-p_2p_3$,
giving an $(L:a)$-coloring of $G$ by the choice of $\psi(p_2)$.  This is a contradiction.
\end{subproof}

\claim{cl-semi5path}{If $T=(a,G,S,L)$ is a canvas close to $T_0$, then $G$ does not contain a semi-neighboring $5$-path of $S$.}
\begin{subproof}
By \refclaim{cl-hasnice}, we can assume that $T$ is nice.
Suppose for a contradiction $P=p_0p_1p_2p_3p_4p_5$ is a semi-neighboring $5$-path of $S$ in $G$.  By symmetry, assume that $\deg(p_2)\ge \deg(p_3)$.
For $i\in\{0,1,4,5\}$,
the neighbor of $p_i$ in $S$ is unique by \refclaim{cl-cors}(ii).
For $i\in \{0,4\}$, let $R_i$ be the set of vertices in $V(G)\setminus (V(S)\cup\{p_i,p_{i+1}\})$ with a neighbor both in $\{p_i,p_{i+1}\}$ and in $S$,
and let $S_i=R_i\cup \{p_i,p_{i+1}\}$.
By \refclaim{cl-cors}(iii) and (iv), observe that $|R_i|\le 1$, $S_0\cap S_4=\emptyset$, and the distance between $S_0$ and $S_4$ is at least two.
By \refclaim{cl-semi3path}, $p_2,p_3\not\in S_0\cup S_4$ and the distance between $S_0$ and $S_4$ in $G-V(S)$ is at least three.
Let $H=G[V(S)\cup S_1\cup S_2\cup \{p_2\}]$.  A straightforward case analysis using \refclaim{cl-1path} and \refclaim{cl-2path} shows that
$d_T(T|H)\ge -4+\deg(p_2)\alpha-5\varepsilon\ge -4+\deg(p_3)\alpha-5\varepsilon$.
Note also that the canvas $T/H$ is normal: $H$ is an induced subgraph, no vertex of $V(G)\setminus V(H)$ can have three neighbors in $S_1\cup S_2\cup\{p_2\}$
since $G$ has girth at least five and the distance between $S_1$ and $S_2$ is at least three
in $G-V(S)$,  no vertex $v\in V(G)\setminus V(H)$ can have more than one neighbor in $S$ by \refclaim{cl-cors}(ii), and $v\in V(G)\setminus V(H)$ cannot have two neighbors
in $S_1\cup S_2\cup\{p_2\}$ and one neighbor in $S$ since $v\not\in R_1\cup R_2$ and by \refclaim{cl-cors}(iii).

Let $\psi$ be an $(L:a)$-coloring of $H$ extending $\psi_L$, in which the color sets of vertices of $V(H)\setminus V(S)$ are chosen as follows.
By \refclaim{cl-adj}, we have $|A_T(p_4)\setminus A_T(p_3)|<a$, and thus we can choose $\psi(p_4)$ as an $a$-element subset of $A_T(p_4)\cap A_T(p_3)$.
Then $|A_T(p_3)\setminus \psi(p_4)|=2a$, and since $|A_T(p_2)|=3a$, we can choose $\psi(p_2)$ as an $a$-element subset of $A_T(p_2)\setminus (A_T(p_3)\setminus \psi(p_4))$.
We extend $\psi$ to other vertices of $H$ greedily, which is possible since $p_2$ and $p_4$ are in different components of the forest $H-V(S)$.
If $\deg(p_3)=3$, then let $G'=G-p_3$, otherwise let $G'=G-p_2p_3$.  If $\deg(p_3)=3$, then $|E(G)\setminus E(G')|>2$ and $d_T(T|H)\ge-4+3\alpha-5\varepsilon=-3 + 6\varepsilon$.
If $\deg(p_3)\ge 4$, then $d_T(T|H)\ge-4+4\alpha-5\varepsilon=-3 + \alpha+6\varepsilon$.  In either case, \refclaim{cl-reduce} implies
$\psi$ extends to an $(L:a)$-coloring $\varphi$ of $G'$.  If $\deg(p_3)=3$, we can furthermore greedily extend $\varphi$ to an $(L:a)$-coloring of $G-p_2p_3$.
However, $\varphi(p_3)\subset A_T(p_3)\setminus \psi(p_4)$, and thus $\varphi(p_3)\cap \psi(p_2)=\emptyset$ by the choice of $\psi(p_2)$.  Hence, $\varphi$ is an $(L:a)$-coloring of $G$,
which is a contradiction.
\end{subproof}

We are now ready to finish the proof of Theorem~\ref{thm-canvas}.  We apply Lemma~\ref{lemma-neiparel} to $T_0$.  Since $T_0$ is non-singular,
\refclaim{cl-cors} excludes the conclusions (1), (2), and (3) of Lemma~\ref{lemma-neiparel}.  Hence, by (4) or (5), there exists a $(\le\!2)$-relaxation $T=(a,G_0,S,L_0)$ of $T_0$
and a semi-neighboring $3$- or $5$-path of $S$ in $G_0$.  By \refclaim{cl-close}, $T$ is close to $T_0$.  However, this contradicts \refclaim{cl-semi3path} or \refclaim{cl-semi5path}.
\end{proof}

Theorem~\ref{thm-hypcyl} now follows by a straightforward application of Euler's formula.
\begin{proof}[Proof of Theorem~\ref{thm-hypcyl}]
Since $G$ has girth at least five and has faces bounded by cycles $C_1$ and $C_2$, Euler's formula gives
\begin{align*}
3|E(G)|-5|V(G)|+|C_1|+|C_2|&\le 3|E(G)|-5|V(G)|+10+\sum_{f\in F(G)} (|f|-5)\\
&=5(|E(G)|-|V(G)|-|F(G)|+2)\le 0.
\end{align*}
Consider the critical canvas $T=(a,G,C_1\cup C_2,L)$.  If $T$ is singular, then $|V(G)|\le |C_1|+|C_2|+1 < 89(|C_1|+|C_2|)$.
Hence, we can assume $T$ is non-singular, and thus by Theorem~\ref{thm-canvas}, we have
\begin{align*}
3&\le d(T)\le 3e(T)-(5+\varepsilon)v(T)\\
&=3(|E(G)|-|C_1|-|C_2|)-(5+\varepsilon)(|V(G)|-|C_1|-|C_2|)\\
&=3|E(G)|-5|V(G)|+(2+\varepsilon)(|C_1|+|C_2|)-\varepsilon|V(G)|\\
&\le (1+\varepsilon)(|C_1|+|C_2|)-\varepsilon|V(G)|.
\end{align*}
Consequently, $|V(G)|<\tfrac{1+\varepsilon}{\varepsilon}(|C_1|+|C_2|)=89(|C_1|+|C_2|)$.
\end{proof}

\end{document}